\documentclass[11pt]{article}

\usepackage[utf8]{inputenc}
\usepackage{mathrsfs}
\usepackage{amscd}
\usepackage{amsmath}
\usepackage{amssymb}
\usepackage{amstext}
\usepackage{amsthm}
\usepackage{bbold}
\usepackage{bm}
\usepackage{color}
\usepackage{framed}
\usepackage[dvips,letterpaper,margin=1in]{geometry}
\usepackage{graphicx}
\usepackage{hyperref}
\usepackage[capitalize]{cleveref}
\usepackage{mathtools}
\usepackage{multirow}
\usepackage{colonequals}
\usepackage{rotating}
\usepackage{setspace}
\usepackage{verbatim}
\usepackage{optidef}
\usepackage[dvipsnames]{xcolor}
\usepackage[sort,numbers]{natbib}
\usepackage[mathmode,centertableaux]{ytableau}
\usepackage{natbib}
\definecolor{cblack}{rgb}{0,0,0}
\definecolor{cblue}{rgb}{0.121569,0.466667,0.705882}    
\definecolor{corange}{rgb}{1.000000,0.498039,0.054902}  
\definecolor{cgreen}{rgb}{0.172549,0.627451,0.172549}   
\definecolor{cred}{rgb}{0.839216,0.152941,0.156863}     
\definecolor{cpurple}{rgb}{0.580392,0.403922,0.741176}  
\definecolor{cbrown}{rgb}{0.549020,0.337255,0.294118}   
\definecolor{cpink}{rgb}{0.890196,0.466667,0.760784}
\definecolor{cgray}{rgb}{0.498039,0.498039,0.498039}
\definecolor{cgreen2}{rgb}{0.7372549019607844, 0.7411764705882353, 0.13333333333333333}

\hypersetup{
  linkcolor  = cblue,
  citecolor  = cgreen,
  urlcolor   = corange,
  colorlinks = true,
}

\newtheorem{Th}{Theorem}[section]
\newtheorem{Prop}[Th]{Proposition}
\newtheorem{Rem}[Th]{Remark}
\newtheorem{Lemma}[Th]{Lemma}

\newtheorem{remark}[Th]{Remark}
\newtheorem{Def}[Th]{Definition}

\newtheorem{Cor}[Th]{Corollary}
\newtheorem{Conj}[Th]{Conjecture}

\DeclareSymbolFont{bbold}{U}{bbold}{m}{n}
\DeclareSymbolFontAlphabet{\mathbbold}{bbold}

\newcommand{\RR}{\mathbb{R}}

\newcommand{\NN}{\mathbb{N}}
\newcommand{\ZZ}{\mathbb{Z}}

\newcommand{\CC}{\mathbb{C}}

\newcommand{\TT}{\mathbb{T}}

\newcommand{\One}{\mathbbold{1}}

\newcommand{\bD}{\bm D}

\newcommand{\bF}{\bm F}

\newcommand{\bH}{\bm H}

\newcommand{\bM}{\bm M}

\newcommand{\bP}{\bm P}

\newcommand{\bS}{\bm S}

\newcommand{\bU}{\bm U}

\newcommand{\bX}{\bm X}
\newcommand{\bY}{\bm Y}

\newcommand{\ba}{\bm a}

\newcommand{\be}{\bm e}
\newcommand{\bbf}{\bm f}

\newcommand{\bv}{\bm v}
\newcommand{\bw}{\bm w}

\newcommand{\bs}{\bm s}

\newcommand{\bx}{\bm x}

\newcommand{\by}{\bm y}

\newcommand{\sB}{\mathcal{B}}

\newcommand{\sH}{\mathcal{H}}
\newcommand{\sI}{\mathcal{I}}

\newcommand{\sM}{\mathcal{M}}

\newcommand{\sR}{\mathcal{R}}
\newcommand{\sS}{\mathcal{S}}

\newcommand{\id}{\mathsf{id}}

\newcommand{\what}{\widehat}

\newcommand{\tr}{\mathrm{tr}}

\newcommand{\sgn}{\mathrm{sgn}}

\newcommand{\spike}[2]
{\bgroup
  \sbox0{#2}%
  \rlap{\usebox0}%
  \hspace{0.5\wd0}%
  \makebox[0pt][c]{\rule[\dimexpr \ht0+1pt]{0.5pt}{#1}}
  \makebox[0pt][c]{\rule[\dimexpr -\dp0-#1-1pt]{0.5pt}{#1}}
  \hspace{0.5\wd0}%
\egroup}

\title{Dual bounds for the positive definite functions approach to mutually unbiased bases}
\date{February 26, 2022}

\author{}
\usepackage{authblk}
\author[1]{Afonso S.\ Bandeira\thanks{Email: \textit{bandeira@math.ethz.ch}}}
\author[1]{Nikolaus Doppelbauer\thanks{Email: \textit{doppelbn@student.ethz.ch}}}
\author[2]{Dmitriy Kunisky\thanks{Email: \textit{dmitriy.kunisky@yale.edu}. Partially supported by ONR Award N00014-20-1-2335, a Simons Investigator Award to Daniel Spielman, and NSF grants DMS-1712730 and DMS-1719545. Part of this work was performed while with New York University.}}

\affil[1]{Department of Mathematics, ETH Z\"{u}rich}
\affil[2]{Department of Computer Science, Yale University}

\begin{document}

\maketitle
\thispagestyle{empty}

\begin{abstract}
    A long-standing open problem asks if there can exist 7 mutually unbiased bases (MUBs) in $\CC^6$, or, more generally, $d + 1$ MUBs in $\CC^d$ for any $d$ that is not a prime power.
    The recent work of Kolountzakis, Matolcsi, and Weiner (2016) proposed an application of the method of positive definite functions (a relative of Delsarte's method in coding theory and Lov\'{a}sz's semidefinite programming relaxation of the independent set problem) as a means of answering this question in the negative.
    Namely, they ask whether there exists a polynomial of a unitary matrix input satisfying various properties which, through the method of positive definite functions, would show the non-existence of 7 MUBs in $\CC^6$. Using a convex duality argument, we prove that such a polynomial of degree at most 6 cannot exist.
    We also propose a general dual certificate which we conjecture to certify that this method can never show that there exist strictly fewer than $d + 1$ MUBs in $\CC^d$.
\end{abstract}

\clearpage

\tableofcontents
\thispagestyle{empty}

\clearpage

\setcounter{page}{1}

\section{Introduction}

Collections of mutually unbiased bases of $\CC^d$ (MUBs) have long been of interest in the combinatorics and geometry of packing problems and find important applications in quantum information and coding theory \cite{bengtsson2007mutually,spengler2012entanglement,WOOTTERS1989363,bennett2020quantum}.
These objects are defined as follows.

\begin{Def}\label{def: MUBs}
    Two orthonormal bases $\mathcal{B} = \{\bv_1, \dots, \bv_d\}$ and $\mathcal{B}^{\prime} = \{\bw_1, \dots, \bw_d\}$ of $\mathbb{C}^d$ are called \emph{unbiased} if for any $\bv \in \mathcal{B}$ and $\bw \in \mathcal{B}^{\prime}$, $\lvert\langle \bv, \bw\rangle\rvert = 1 / \sqrt{d}$.
    A set of orthonormal bases $\{ \mathcal{B}_1, \dots \mathcal{B}_n \}$ any two of which are unbiased is called \emph{mutually unbiased}. 
\end{Def}
\noindent
Despite a rich literature studying MUBs, many basic questions remain, including how many MUBs can exist in many dimensions $d$.
Specifically, the number
\[
    N(d) \colonequals \max\{n: \text{there exist } n \text{ mutually unbiased bases in } \CC^d \}
\]
is not known for any $d \geq 2$ that is not a prime power.
In particular, even the smallest non-trivial case $N(6)$ remains unknown.

A general bound, which in the sequel we will call the \emph{Welch bound}, shows that $N(d) \leq d + 1$; however, it has long been believed that $N(6) < 7$ strictly and indeed that perhaps $N(6) = 3$ (see, e.g., \cite{zauner1999grundzuge,PhysRevA.79.052316} and our discussion below).
The recent work \cite{KMW-2018-PositiveDefiniteMUBs} proposed a technique, based on the \emph{method of positive definite functions}, to improve on the Welch bound.
This method may be viewed as an analogue of semidefinite programming relaxation that applies even to infinite-dimensional settings.
In this case, it shows that certain positive definite functions on the unitary group $U(d)$ yield upper bounds on $N(d)$; optimizing this upper bound over \emph{all} positive definite functions is an infinite-dimensional convex optimization problem with some value $B(d) \geq N(d)$.

In this paper, we consider the dual of this optimization problem, and propose a dual certificate that, if feasible for the dual program, would show $B(6) = 7$ (and the general version of the same construction would show that $B(d) = d + 1$ for all $d$); that is, feasibility would imply that the method of positive definite functions cannot improve on the Welch bound, answering in the negative a question posed in~\cite{KMW-2018-PositiveDefiniteMUBs}.
The feasibility of our certificate is only conjectural, but, as a first step, we prove its feasibility for a restricted program, showing that positive definite low-degree polynomials cannot improve on the Welch bound $N(6) \leq 7$.

\subsection{Mutually unbiased bases and Hadamard matrices}

Towards stating our main result, we first recall some standard facts about MUBs and the closely related \emph{Hadamard matrices}. 
First, the bound alluded to above is as follows.
We include the proof for the sake of completeness.

\begin{Prop}[Welch bound; e.g., Section 4 of \cite{bengtsson2007three}] 
    \label{Welch}
    For any $d \geq 2$, there are at most $d + 1$ mutually unbiased bases in $\CC^d$, i.e., $N(d) \leq d + 1$.
\end{Prop}

\begin{proof}
The general Welch bound states that, for any $k \geq 1$ and $X \subset \CC^d$ with $\lvert X \rvert = n$, we have
\[
    \binom{n+k-1}{k}\sum_{\bx,\by \in X} \lvert \langle \bx, \by \rangle\rvert^{2k} \geq \left(\sum_{\bx \in X} \|\bx\|^{2k}\right)^2.
\]
Applying this with $k=2$ and $X = \cup_{i=1}^n \mathcal{B}_i $ shows that $n \leq d+1$, i.e., that there can be at most $d+1$ MUBs in $\mathbb{C}^d$.
\end{proof}
\noindent
Explicit constructions (see the references below for details) show that this bound is tight in the case of prime powers.
\begin{Prop}[Prime power dimensions \cite{combescure2007circulant,combescure2009block,10.1007/978-3-540-24633-6_10}] \label{prop:prime-power-dim}
    For $p$ prime and $k \geq 1$, $N(p^k) = p^k + 1$.
\end{Prop}
\noindent
Finally, the main general tool for producing MUBs in other dimensions is the following construction based on taking tensor products of bases.
\begin{Prop}[Tensor product construction; Lemma 3 of \cite{10.1007/978-3-540-24633-6_10}]
    \label{prop:tensor}
    For any $d,d' \geq 1$ we have $N(dd^{\prime}) \geq \min\{N(d), N(d^{\prime})\}$.
\end{Prop}
\noindent
In particular, this yields the lower bound $N(6) \geq 3$ which was conjectured by Zauner to be optimal in 1999 \cite{zauner1999grundzuge,bengtsson2007three}.
The problem of improving either the upper or lower bounds in $3 \leq N(6) \leq 7$ has since remained open.
We note, however, that the bound produced by Proposition~\ref{prop:tensor} is suboptimal in certain other dimensions, as witnessed by the construction of \cite{wocjan2004new} based on Latin squares.
For these and other concrete constructions of MUBs in low dimension, the reader may consult \cite{brierley2010mutually}.

The connection between MUBs and Hadamard matrices, which we define below, lets us formulate the problem of finding MUBs as a packing problem in the unitary group $U(d)$.

\begin{Def}[Hadamard matrices]
    We call a unitary matrix $\bm{H} \in U(d)$ a \emph{Hadamard matrix} if $\lvert H_{i,j} \rvert = 1 / \sqrt{d}$ for all $i,j \in [d]$.
    We denote the set of Hadamard matrices by $H(d) \subset U(d)$.
\end{Def}

Let us associate to a basis $\sB$ of $\CC^d$ a unitary matrix $\bU$ whose columns are the vectors of $\sB$; we will equate these two objects freely below.
Then, the following equivalence is immediate.

\begin{Prop}\label{MUBs differences}
    A set of bases $\{ \mathcal{B}_1, \dots \mathcal{B}_n \}$ of $\mathbb{C}^d$ is mutually unbiased if and only if the associated matrices $\{ \bm{U}_1,\bm{U}_2, \dots ,\bm{U}_n \}$ satisfy $\bm{U}_i^*\bm{U}_j \in H(d)$ for all $i \neq j$.
\end{Prop}

\subsection{Upper bounds from positive definite functions}\label{sec: upper bound}
We next review the technique proposed in \cite{KMW-2018-PositiveDefiniteMUBs} for bounding $N(d)$ using positive definite functions.
Let us work for the moment in the general setting of \emph{packing problems} over a compact group $G$, where we search for collections of elements $g_1, \dots, g_n \in G$ such that the pairwise differences $g_i^{-1}g_j$ avoid some forbidden subset $A \subset G$.
By Proposition~\ref{MUBs differences}, the problem of finding MUBs falls under this framework with $G = U(d)$ and $A = H(d)^c$.
In such a setting, we write $e$ for the identity element of $G$, and $\nu$ for the Haar measure on $G$ normalized so that $\nu(G) = 1$.

The following is the key definition underlying the technique we will study.
\begin{Def}[Positive definite function] 
    For $G$ a compact group, we say that a continuous $f: G \to \CC$ is \emph{positive definite} if, for all $n \geq 1$, $g_1, \dots, g_n \in G$, and $c_1, \dots , c_n \in \mathbb{C}$, we have
    \[
        \sum_{i,j=1}^n \overline{c_i}f(g_i^{-1}g_j)c_j \geq 0.
    \]
    We denote the set of such functions by $C_{\succeq 0}(G)$.
\end{Def}

The main tool used by \cite{KMW-2018-PositiveDefiniteMUBs} is the following general application of positive definite functions to packing problems in compact groups.
\begin{Th}[Theorem 2.3 in \cite{KMW-2018-PositiveDefiniteMUBs}]\label{th: upper bound}
Let $G$ be a compact group, and let $A = A^{-1} \subset G$ with $e \in A$. Suppose $h: G \rightarrow \mathbb{R}$ satisfies the following properties:
\begin{enumerate}
    \item $h \in C_{\succeq 0}(G)$,
    \item $h(x) \leq 0$ for all $x \in A^c$, and
    \item $\int h d \nu > 0$.
\end{enumerate}
Then, for any $B \subset G$ such that $b^{-1}c \in A^c$ for all $b, c \in B$ distinct, we have
\[
    \lvert B \rvert \leq \frac{h(e)}{\int h d \nu} .
\]
\end{Th}
\noindent
Let us write the optimal bound achieved in this way, for a given $G$ and $A$, as
\[
    B(G, A) \colonequals \left\{\begin{array}{ll} \text{infimum of} & h(e) \\[0.5em] \text{subject to} & h \in C_{\succeq 0}(G) \text{ real-valued}, \\[0.5em] & h(g) \leq 0 \text{ for all } g \in A^c, \\[0.6em] & \int h d\nu = 1 \end{array}\right\},
\]
where we assume without loss of generality in the optimization that $h$ is normalized to fix its average value.

We may write the application of this result to the problem of MUBs by setting $G = U(d)$ and $A^c = H(d)$, obtaining the following.
\begin{Cor}
    Let $\nu$ denote the normalized Haar measure of $U(d)$.
    Then,
    \begin{align} \label{Primal}
        N(d) \leq B(d) &\colonequals \left\{\begin{array}{ll} \text{infimum of} & h(\bm I) \\[0.5em] \text{subject to} & h \in C_{\succeq 0}(U(d)) \text{ real-valued}, \\[0.5em] & h(\bm H) \leq 0 \text{ for all } \bm H \in H(d), \\[0.6em] & \int h d\nu = 1 \end{array}\right\} \\
        &= B(U(d), H(d)^c). \nonumber
    \end{align}
\end{Cor}

Of course, the question remains to identify $B(d)$.
In \cite{KMW-2018-PositiveDefiniteMUBs}, the authors show that the polynomial
\begin{equation} \label{h0-poly}
        h_0(\bm{U}) =   \sum_{i,j =1}^d \lvert U_{i,j} \rvert^4 -1
\end{equation}
has $h_0 / \int h_0 d\nu$ feasible for \eqref{Primal} (note that in fact $h_0(\bm U) = 0$ for all $\bm U \in H(d)$), and has $h_0(\bm I) = d - 1$ and $\int h_0 d \nu = \frac{d-1}{d+1}$.
Therefore,
\[ B(d) \leq \frac{h_0(\bm I)}{\int h_0 d\nu} = \frac{d - 1}{\frac{d - 1}{d + 1}} = d + 1, \]
giving an alternative proof of the Welch bound (our Proposition~\ref{Welch}).
They further conjecture that, when $d = 6$, it may be possible to find a polynomial feasible for \eqref{Primal} which would show that $B(6) < 7$ (and therefore that $N(d) < 7$).
Moreover, they conjecture a specific set of polynomials to vanish on $H(6)$, which could then be used to produce a better upper bound.

\subsection{Main result}

Let us write $\CC_{\leq k}[\bm U, \overline{\bm U}]$ for the set of polynomials in the entries of a matrix $\bm U$ and their conjugates whose degree is at most $k$.
For example, the polynomial given in \eqref{h0-poly} belongs to $\CC_{\leq 4}[\bm U, \overline{\bm U}]$.
Following \cite{KMW-2018-PositiveDefiniteMUBs} in restricting our attention to polynomials, we will consider the related optimization problems
\[ B(d) \leq B_{\leq k}(d) \colonequals \left\{\begin{array}{ll} \text{infimum of} & h(\bm I) \\[0.5em] \text{subject to} & h \in C_{\succeq 0}(U(d)) \cap \CC_{\leq k}[\bm U, \overline{\bm U}] \text{ real-valued}, \\[0.5em] & h(\bm H) \leq 0 \text{ for all } \bm H \in H(d), \\[0.6em] & \int h d\nu = 1 \end{array}\right\}, \]
which is identical to \eqref{Primal} except for further constraining $h \in \CC_{\leq k}[\bm U, \overline{\bm U}]$. Writing $\CC[\bm U, \overline{\bm U}] \colonequals \cup_{k \in \NN} \CC_{\leq k}[\bm U, \overline{\bm U}]$, an application of the Stone-Weierstrass theorem yields
\[
    B(d) = \lim_{k \to \infty} B_{\leq k}(d).
\]

In this notation, the general conjecture proposed in \cite{KMW-2018-PositiveDefiniteMUBs} is that $B_{\leq k}(6) < 7$ for some modest $k$. Our main result, giving evidence against this conjecture, is the following.
\begin{Th}
    \label{th:lower bound}
    Positive definite polynomials of degree at most 6 cannot show there are less than 7 MUBs in $\mathbb{C}^6$, i.e., $B_{\leq 6}(6) = 7$.
\end{Th}
\noindent

It remains unclear if our technique extends to larger values of $d$ and $k$, but we propose a dual certificate construction that, if verified, would yield the following stronger negative result.

\begin{Conj}\label{conj: strong lower bound}
    The method of positive definite functions cannot improve on the Welch bound in any dimension, i.e., for all $d \geq 1$, $B(d) = d + 1$.
\end{Conj}
\noindent

In Section~\ref{sec:final-verification} we present the concrete linear-algebraic conjectures that would yield the results $B_{\leq k}(d) = d + 1$ for various choices of $k$ and $d$.
In principle, each such result for fixed $k$ and $d$ would be implied by a finite symbolic computation; we carry this out with computer assistance for $k = 6$ and $d = 6$, but these computations quickly become computationally intractable for larger values.
We leave it as an open problem to find more conceptual proofs of our linear-algebraic claims.

\subsection{Related work}

\subsubsection{Convex relaxations of packing problems}
Methods based on convex optimization give state-of-the-art upper bounds for a range of packing problems.
This idea was pioneered by Delsarte \cite{Delsarte-1972-LPUnrestrictedCodes,Delsarte-1973-LPAssociationSchemes}, who used linear programming relaxations to prove packing bounds in coding theory.
A similar idea based instead on semidefinite programming was introduced by Lov\'{a}sz through his ``$\vartheta$ function'' bound for the independent set problem in finite graphs in the seminal work \cite{Lovasz-1979-ShannonCapacityGraph}.
A related method closer to the setting we consider, for the non-compact but abelian group $\RR^d$ (in which case the semidefinite programs we consider become linear programs as in Delsarte's setting), was proposed for the sphere packing problem in \cite{Cohn_2003} and is an essential ingredient in the recent breakthrough results identifying exactly optimal sphere packings in dimension 8 and 24 \cite{Viazovska-2017-SpherePacking8,Cohn_2017}.
An overview of other applications is given in Table~1.5 of \cite{de2015semidefinite}; in that framework, the kind of bound we study is an instance of a ``two-point semidefinite programming bound'' on a packing problem.

The task of producing dual bounds on such techniques has not been studied as much.
As a first step, \cite{de2015semidefinite} proposed a unified framework in the style of the sum-of-squares hierarchy for such duality relations for packing problems.
In a more concrete application, the notable recent work \cite{cohn2021dual} studies the dual of the linear programs of \cite{Cohn_2003}, showing a significant gap between the densest known sphere packings and the upper bounds achievable by linear programs in some dimensions.
Their proof technique takes advantage of the remarkable structure and symmetry of modular forms in these special dimensions.
The analogous construction in our setting will be based on the Fourier matrix as presented in Section~\ref{sec:fourier}, and we discuss some intriguing combinatorial and number-theoretic open questions raised by this construction in Section~\ref{sec:final-verification}.

\subsubsection{Other relaxations for the MUB problem}

Several formulations of upper bounds for MUBs based on convex relaxations besides that of \cite{KMW-2018-PositiveDefiniteMUBs} have been proposed in the literature.
The work \cite{matolcsi2010fourier} uses that the existence of MUBs is related to the existence of MUBs further restricted so that each unitary matrix $\bU_i$ is in fact also a Hadamard matrix, and then, viewing the columns of these matrices as separate variables in a packing problem, applies Theorem~\ref{th: upper bound} on the torus to obtain the Welch bound.
This method is further extended in \cite{matolcsi2013systems} to derive various constraints on MUBs in low dimension ($d \leq 5$) and, with weaker results, in arbitrary dimension. 

Another line of work, initiated by \cite{navascues2012sdp} and continued more recently by \cite{GP-2021-MUBOptimizationSymmetry}, uses that computing the size of a MUB may also be cast as an optimization problem over non-commutative (i.e., matrix-valued) variables. Such problems admit a variant of the sum-of-squares hierarchy of convex relaxations, which these works use to formulate different semidefinite programming relaxations.

It would be interesting to investigate both dual bounds in these other frameworks and the possibility that these various approaches might be equivalent in power.
We leave these directions to future work.

\subsection{Organization}

The remainder of the paper is organized as follows.
In Section~\ref{sec:duals} we derive dual programs both for the general problem of optimizing upper bounds on packing problems over compact groups and for the problem of MUBs.
In doing so, we introduce some further background on positive definite measures and the Fourier transform of a measure over a compact group.
In Section~\ref{sec:fourier} we motivate and define the dual certificate we propose for the problem of MUBs.
In Section~\ref{sec: d=6} we prove Theorem~\ref{th:lower bound} using this dual certificate.

\subsection{Notation}

We use the standard notation $[d] = \{1, \dots, d\}$.
For a proposition $P$, we write $\One\{P\}$ for the \emph{indicator} of $P$, which takes the value 1 if $P$ holds and 0 if $P$ does not hold.

We write $U(d)$ and $SU(d)$ for the unitary and special unitary groups of $d \times d$ complex matrices, respectively, $S_d$ for the symmetric group on $[d]$, $\TT^d$ for the torus group or $d$-fold direct product of the complex unit circle with itself, and $G \leq H$ and $G < H$ for the relations of $G$ being a subgroup and a proper subgroup of $H$, respectively.

In part of our discussion, we will write $G$ for an arbitrary compact group.
In this context, we also write $e \in G$ for the identity element and $\nu$ for the Haar measure normalized so that $\nu(G) = 1$.
We write $\sI(G)$ for the collection of non-isomorphic irreducible representations (or \emph{irreps}) of $G$.
We write $1 \in \sI(G)$ for the irreducible trivial representation, and do not introduce any special notation for isomorphism of representations, writing, e.g., $\pi = 1$ for a representation $\pi$ being isomorphic to the trivial representation.

We use boldface lowercase ($\bx$) for vectors and boldface uppercase ($\bX$) for matrices; scalar entries of either are written in the same case but in regular font ($x_i, X_{ij}$).
For a permutation $\sigma \in S_d$, we write $\bm P_{\sigma} \in \{0, 1\}^d$ for the associated permutation matrix, so that $(\bm P_{\sigma} \bx)_i = x_{\sigma(i)}$.

\section{Duals of positive definite function programs}
\label{sec:duals}

\subsection{Packing problems over compact groups}
\label{sec:dual-general}

We now discuss how to obtain dual bounds on optimizations such as those prescribed by Theorem~\ref{th: upper bound}.
To do this, we formulate that optimization as a conic optimization problem.
We keep for the moment to the general setting of \cite{KMW-2018-PositiveDefiniteMUBs} over compact groups, and later will reformulate our main result for the special case of MUBs.

Since in Theorem~\ref{th: upper bound} we may without loss of generality assume that $\int h d\nu = 1$, optimizing the upper bound given by Theorem~\ref{th: upper bound} can be rewritten as
\begin{equation} \label{eq: mub mini}
    B(G, A) = \left\{\begin{array}{ll} 
    \text{infimum of} & h(\bm I) \\ [0.5em]
    \text{subject to} & h \in K \colonequals C_{\succeq 0}(G) \cap C_{\leq 0, A^c}(G), 
    \\[0.6em] & \int h d\nu = 1
    \end{array}\right\}.
\end{equation}
Here $ C_{\succeq 0}(G) \subset C(G)$ is the convex cone of positive definite functions and $C_{\leq 0, A^c}(G) \subset C(G)$ is the convex cone of functions which are non-positive on $A^c$.
In this form, \eqref{eq: mub mini} is a conic program, albeit one over an infinite-dimensional space.
The reader may consult \cite{barvinok2002course} for discussion of convex optimization and conic programs in this general setting.

To proceed towards producing dual certificates for such a program, we note that the dual space of $C(G)$ is the space $\sM(G)$ of Radon measures\footnote{Radon measures are signed measures on the Borel $\sigma$-algebra satisfying certain regularity properties; for the full definition, see Chapter 7 of \cite{folland1999real}.} on $G$, with the dual pairing $\langle f , \mu \rangle \colonequals \int f d \mu$.
The dual program is then
\begin{equation} \label{eq: mub maxi}
    B^*(G, A) = \left\{\begin{array}{ll} 
    \text{supremum of} & c \in \RR \\[0.5em] 
    \text{subject to} & \delta_e - c\nu \in K^*
    \end{array}\right\}.
\end{equation}
Here $K^*$ denotes the dual cone of $K$ and $\delta_e$ is the Dirac measure at the identity.

It is simple to give a hands-on proof of \emph{weak duality}, i.e., $B(G, A) \geq B^*(G, A)$.
This is because, if $f$ is feasible for \eqref{eq: mub mini} and $c$ for \eqref{eq: mub maxi}, then by the definition of $K^*$ we have
\[
    0 \leq \langle f, \delta_e - c\nu \rangle = f(e) - c.
\]
Thus to prove lower bounds on \eqref{eq: mub mini}, it suffices to produce feasible $c$ for \eqref{eq: mub maxi}.

To study the dual program, instead of $K^*$, we use the simpler cone
\[
    K^* = (C_{\succeq 0}(G) \cap C_{\leq 0, A^c}(G))^* \supset C_{\succeq 0}(G)^* + C_{\leq 0, A^c}(G)^*.
\]
Thus, to show $\delta_e - c\nu \in K^*$, it suffices to produce $\lambda \in C_{\succeq 0}(G)^*$ and $\mu$ an unsigned measure
with support in $\overline{A^c}$ (which must then belong to $-C_{\leq 0, A^c}(G)^*$) which fulfill
\[
    \delta_e - c\nu = \lambda - \mu.
\]
We rewrite this as 
\[
    \delta_e + \mu \succeq  c\nu,
\]
which means
\begin{equation} \label{eq:dual-feas-cond}
    \delta_e + \mu - c\nu \in C_{\succeq 0}(G)^* \equalscolon \mathcal{M}_{\succeq 0}(G).
\end{equation}

We next give the characterization of this convex cone of measures.
First, we define the Fourier transform of a measure on $G$.
We recall that $\sI(G)$ denotes the irreducible representations (irreps) of a compact group $G$ distinct up to isomorphism. Standard results of representation theory imply that all such irreps are finite-dimensional and that any arbitrary representation splits into a (possibly infinite) direct sum of irreps.

\begin{Def}[Fourier transform of a measure]
    For $\lambda \in \sM(G)$ and $\pi$ a unitary representation of $G$, we define
    \[ \what{\lambda}(\pi) \colonequals \int_G \pi(g) d\lambda(g). \]
\end{Def}

The following is the key characterization of positive definite functions on compact groups.
This follows for instance by combining Proposition 3.35 and Theorem 3.20 in \cite{folland2016course}.
\begin{Th}\label{th: pos def rep}
    $f \in C_{\succeq 0}(G)$ if and only if, for some unitary representation $\pi$ of $G$ on a Hilbert space $\mathcal{H}$  and some $v \in \mathcal{H}$, for all $g \in G$ we have $f(g) = \langle \pi(g) v, v \rangle$.
\end{Th}

\noindent
The following characterization is then an analog of Bochner's theorem for compact groups.

\begin{Th} \label{thm:pos-def-measures}
    $\sM_{\succeq 0}(G)$ is the set of $\lambda \in G$ with $\what{\lambda}(\pi) \succeq \bm 0$ for all $\pi \in \sI(G)$.
\end{Th}

\begin{proof}
   Theorem~\ref{th: pos def rep} tells us that a function $f(g)$ is in $C_{\succeq 0}(G)$  if and only if it is of the form $\langle \pi(g) v, v \rangle$. Integrating against $\lambda$ then gives
   \[
        \int \langle \pi(g) v, v \rangle d \lambda(g) = \langle \widehat{\lambda}(\pi) v, v \rangle,
   \]
   which shows that $\lambda \in \sM_{\succeq 0}(G)$ if and only if  $\widehat{\lambda}(\pi) \succeq \bm 0$ for any representation $\pi$. Since any such $\pi$ is a direct sum of irreducible subspaces, this statement is in turn equivalent to $\what{\lambda}(\pi) \succeq \bm 0$ for all $\pi \in \sI(G)$.
\end{proof}

\begin{Rem}[Hermitian Fourier transform]
    We note that the definition of positive definiteness presumes that $\what{\lambda}(\pi)$ is also Hermitian for all $\pi \in \sI(G)$.
    The Fourier transform being Hermitian turns out to be equivalent to a suitable notion of self-adjointness for the measure $\lambda$ itself, under an involution $\lambda \mapsto \lambda^*$ which makes $\sM(G)$ into a Banach-$*$ algebra.
    We discuss this further in Appendix~\ref{app:star-algebra}.
    In the sequel we will only be concerned with real-valued measures for which this is equivalent to $\mu$ being invariant under inversion, that is, $\mu$ must be equal to its pushforward under the map $g \mapsto g^{-1}$.
\end{Rem}

\noindent
It is also simple to compute the Fourier transforms of the two specific measures appearing above: the Fourier transform of the Haar measure is given by $\what{\nu}(\pi) = \One\{\pi = 1\}$, and that of the Dirac mass is given by $\what{\delta_e}(\pi) = \bm I_{\dim(\pi)}$.

We are now equipped to identify when \eqref{eq:dual-feas-cond} will hold for a given $\mu$.
By the above computations of Fourier transforms, we must consider two cases: when $\pi = 1$ is the trivial representation, then the corresponding condition is $\mu(G) \geq c - 1$.
Otherwise, the condition is simply that $\what{\mu}(\pi) \succeq -\bm I$.
Reorganizing these conditions, we find the following formulation of weak duality.

\begin{Lemma}
    \label{lemma: smallest eigenvalue general}
    If $\mu \in \sM(G)$ a probability measure with $\mathrm{supp}(\mu) \subseteq \overline{A^c}$ and $c > 0$ are such that for any $\pi \in \sI(G)$ we have $\widehat{\mu}(\pi) = \int \pi d\mu \succeq -c\bm I$, then
    \[ B(G, A) \geq 1 + \frac{1}{c}. \]
\end{Lemma}

\subsection{Duals over finite-dimensional subspaces of functions}
\label{sec:dual-finite}

The space of positive definite functions on a compact group is an infinite-dimensional vector space.
In practice, to try to produce concrete positive definite function bounds, we would search for ``good'' positive definite functions in some convenient subspace, for example, the space of low-degree polynomials. 

These infinitely many degrees of freedom in the primal program are reflected in the infinitely many constraints on $\mu$ in the dual program---one for each irrep of $G$.
Fortunately, if we restrict our attention to a finite-dimensional subspace of positive definite functions, then there is an associated dual program with likewise finitely many positive semidefinite constraints on $\mu$.

The following is a natural class of positive definite functions to restrict our attention to in the primal program: given a list of irreps $\sR = \{\pi_0 = 1, \pi_1, \dots, \pi_m\}$ of $G$ and corresponding positive semidefinite matrices $\{\bX_0 = 1, \bX_1, \dots, \bX_m\}$ such that $\bX_i \in \CC^{\dim(\pi_i) \times \dim(\pi_i)}$, we may take
\[ h(g) \colonequals \sum_{i = 0}^m \langle \bX_i, \pi_i(g) \rangle = 1 + \sum_{i = 1}^m \tr(\pi_i(g) \bX_i). \]
Functions of this form can be seen to be the intersection of the cone $C_{\succeq 0}(G)$ and the finite dimensional vector subspace of $C(G)$ spanned by the matrix-coefficients of the $\pi_i$.
We will see below that, in fact, in the MUB setting we may choose the $\pi_i$ appropriately to allow $h$ to be an arbitrary positive definite low-degree polynomial.

Note that, in general, we would have $\int h d\mu = \tr(\bX_0)$, so our choice $\bX_0 = 1$ ensures that $h$ satisfies the same normalization as in our definition of $B(G, A)$.
Thus, using the notation $\mathcal{S}_{\ell}^+ \subset \mathbb{C}^{\ell \times \ell}$ for the cone of positive semidefinite matrices, one can see that under this restriction the optimization \eqref{eq: mub mini} transforms into:
\begin{align} \label{eq:primal-finite}
    B_{\sR}(G, A) &\colonequals \left\{\begin{array}{ll} \text{infimum of} & 1 + \sum_{i = 1}^m \tr(\bX_i) \\[0.5em] \text{subject to} & 1 + \sum_{i = 1}^m \tr(\pi_i(g) \bX_i^{\top}) \leq 0 \text{ for all } g \in A^c, \\[0.5em] & \bX_i \in \sS_{\dim(\pi_i)}^+ \text{ for all } i \in [m] \end{array}\right\} \\
    &\geq B(G, A). \nonumber
\end{align}
This is almost an ordinary semidefinite program, except that if $A^c$ is infinite (as in the MUB setting), there are an infinite number of inequality constraints.

The dual program then allows more measures; in particular,
if we put $\lambda = \delta_e + \mu$, then the dual program simply optimizes over measures for which the finitely many Fourier coefficients $\widehat{\lambda}(1),\widehat{\lambda} (\pi_1), \dots, \widehat{\lambda}(\pi_m)$ are positive definite matrices, since these are precisely those measures for which we have
\begin{align*}
    \langle f, \lambda \rangle = \int \sum_{i=0}^m \tr( \pi_i(g) \bm{X}_i^\top ) d \lambda(g)  =  \sum_{i=1}^m \tr( \widehat{\lambda}(\pi_i)\bm{X}_i^\top) \geq 0
\end{align*}
for any $f$ feasible for \eqref{eq:primal-finite}.
Thus an analogue of Lemma~\ref{lemma: smallest eigenvalue general} holds after computing the dual.

\begin{Lemma}\label{four coeff feasible finitized}
    Let $\pi_0 = 1, \pi_1, \dots, \pi_m$ be distinct irreps of $G$.
    If $\mu \in \sM(G)$ is a probability measure with $\mathrm{supp}(\mu) \subseteq \overline{A^c}$ such that $\widehat{\mu}(\pi_i) = \int \pi_i(g) d\mu(g) \succeq -c\bm I$ for all $i \in [m]$, then
    \[ B_{\{1, \pi_1, \dots, \pi_m\}}(G, A) \geq 1 + \frac{1}{c}. \]
\end{Lemma}

\subsection{Duals for the problem of MUBs}

Let us now state the consequences of the previous two sections in the specific setting of MUBs, where we take $G = U(d)$ and $A = H(d)^c$. We first apply the results of Section~\ref{sec:dual-general}, which applied to general positive definite functions.

\begin{Lemma}\label{lem:dual-mub}
    If $\mu \in \mathcal{M}(G)$ is a probability measure with $\mathrm{supp}(\mu) \subseteq H(d) $ such that for any irrep $\pi$ of $U(d)$ we have $\widehat{\mu}(\pi) = \int \pi(\bU) d\mu(\bU) \succeq -c \bm I$ for $c > 0$, then
    \[ B(d) \geq 1 + \frac{1}{c}. \]
\end{Lemma}
\noindent
In particular, if one could find such a measure $\mu$ for $c = 1 / d$, then Conjecture~\ref{conj: strong lower bound} would be proven.

We may also give an analogous result for $B_{\leq k}(d)$, the restriction of the positive definite function bound to polynomials of degree at most $k$.
Here, we must introduce a small amount of further representation theory.
Note that $U(d)$ admits a $d$-dimensional representation in which $\bU \in U(d)$ simply acts by matrix multiplication.
We denote this representation by $\rho$; it is irreducible and is often called the \emph{natural representation} of $U(d)$.
This also admits a dual representation $\rho^*$, in which $\bU \in U(d)$ acts by matrix multiplication by $\overline{\bU}$. The linear span of the matrix-coefficients of the representations $\rho^{\otimes r} \otimes (\rho^*)^{\otimes r^{\prime}}$ with $r, r^{\prime} \geq 0$ and $r + r^{\prime} \leq k$ then coincides with $\CC_{\leq k}[\bm U, \overline{\bm U}]$, giving us the following.

\begin{Lemma}\label{lem:dual-mub-finite}
    Let $k \geq 1$, and let $\pi_0 = 1, \pi_1, \dots, \pi_m$ be a list of the distinct irreps appearing in the decomposition of $\rho^{\otimes r} \otimes (\rho^*)^{\otimes r^{\prime}}$ for any $r, r^{\prime} \geq 0$ with $r + r^{\prime} \leq k$.
    If $c > 0$ and $\mu \in \sM(G)$ is a probability measure with $\mathrm{supp}(\mu) \subseteq H(d)$ such that, for all $i \in [m]$, $\widehat{\mu}(\pi_i) \succeq -c \bm I$, then
    \[ B_{\leq k}(d) \geq 1 + \frac{1}{c}. \]
\end{Lemma}

\section{The Fourier dual certificate}
\label{sec:fourier}

We now present our dual certificate construction, which is based on a particular class of Hadamard matrices that we describe below.

\subsection{More on Hadamard matrices}
\label{sec:prelim-hadamard}

The classification of MUBs is closely related to the classification of Hadamard matrices, which also contains many open questions for low dimensions.
Still, we will make use of a basic reduction common in their study.
While real Hadamard matrices are discrete objects since their entries have only two possible values, complex Hadamard matrices come in continuous families.
For this reason, the following notion of equivalence has become widely used in the literature.

\begin{Def}[Equivalent Hadamard matrices]\label{def: equiv had}
    We call two Hadamard matrices $\bm{H},\bm{H}'$ \emph{equivalent}, written $\bm{H} \sim \bm{H}'$, if there are permutation matrices $\bm{P},\bm{P}'$ and unitary diagonal matrices (i.e., ones with diagonal entries of unit norm) $\bm{D},\bm{D}'$, such that
    \[
        \bm{H}' = \bm{PDHD'P'} .
    \]
    We denote by $[\bm H]$ the equivalence class of $\bm H$ under this equivalence relation.
\end{Def}
\noindent
The online resource  \cite{cite_key} lists families of Hadamard matrices found to date in low dimensions. For dimensions $d =2,3,5$ there is only one equivalence class of Hadamard matrices, and for $d = 4$ there is one continuous family of them.
In parallel to the difficulties with identifying MUBs, the set of Hadamard matrices is not well-understood for $d = 6$: several continuous families have been found, but it is not known whether further families exist.

Let us be more precise about the group action under which we consider equivalence classes of Hadamard matrices.
We will later see that the group defined below corresponds to an important symmetry of the convex program we will be interested in solving.
\begin{Def}[Generalized permutations] \label{def:gen-perm}
We write $\TT^d$ for the $d$-fold direct product of the group formed by the unit circle of $\CC$ with itself, and $S_d$ for the symmetric group of permutations of $d$ elements.
Each of these may be identified with a subgroup of $U(d)$, with $\bm\lambda \in \TT^d$ corresponding to the diagonal matrix $\bD_{\bm\lambda}$ and $\sigma \in S_d$ corresponding to the permutation matrix $\bP_{\sigma}$.
We identify the group of products $\bP_{\sigma} \bD$ for $\sigma \in S_d$ and $\bD \in \TT^d$, a closed subgroup of $U(d)$, with the semidirect product $\TT^d \rtimes S_d$, and call it the subgroup of \emph{generalized permutations}.
\end{Def}

We also identify a concrete example of a Hadamard matrix in all dimensions that will play an important role in our main construction.
We note that, per our above discussion, this matrix generates the equivalence class of all Hadamard matrices for $d = 2, 3, 5$.
\begin{Def}[Fourier matrix]\label{ex: fourier}
    The \emph{Fourier matrix}, denoted $\bm{F}_d \in H(d)$, is the matrix with entries 
    \[
        (\bm{F}_d)_{jk} = \frac{1}{\sqrt{d}}\exp\left(\frac{2 \pi i }{d} jk\right).
    \]
    For this matrix we adopt the convention of starting indices from zero: $j, k \in \{0, 1, \dots, d - 1\}$.
\end{Def}

Finally, we introduce the operation of averaging a measure over the equivalence classes of Hadamard matrices, which we define below.
The symmetries of the MUB problem ensure that, in fact, we may restrict our attention to mixtures of such measures rather than arbitrary measures supported on $H(d)$; for details on this reduction see Appendix~\ref{sec: symmetries}.

\begin{Def}[Uniform measure on an equivalence class]
    Let $\bm H \in U(d)$ be a Hadamard matrix.
    We then denote by $\mu_{[\bH]}$ the measure integrating functions in $C(G)$ by
    \[
        \int f d \mu_{[ \bm H]} = \left(\frac{1}{d!}\right)^2\sum_{\sigma \in S_d} \sum_{\sigma' \in S_d}
        \left(\frac{1}{(2\pi)^d}\right)^2\int_{\TT^d} \int_{\TT^d} f( \bm P_\sigma \bm D_{\bm\lambda} \bm H  \bm D_{\bm\lambda^{\prime}} \bm P_{\sigma'})\, d \bm \lambda d \bm\lambda^{\prime},
    \]
    where we recall that $\bD_{\bx}$ denotes the diagonal matrix with diagonal entries given by $\bx$ and where $\int_{\TT^d}F(\bm\lambda)d\bm\lambda$ denotes an integral with respect to the Haar measure on $\TT^d$ with the ``surface area'' normalization so that $\int_{\TT^d} d\bm\lambda = (2\pi)^d$.
\end{Def}
\noindent
The measure $\mu_{[\bm H]}$ ``averages'' over the equivalence class $[\bm{H}]$, and one can check that $\mu_{[\bm H]} = \mu_{[\bm H']}$ if $\bm H \sim \bm H'$, justifying the notation.

\subsection{Intuition from dimensions \texorpdfstring{$d \leq 5$}{d <= 5}}\label{sec: lower dim}

It is instructive to begin by looking at dual certificate measures $\mu$ that are optimal for the dual program in the dimensions $d = 2, 3, 4, 5$ where we both know the value of $N(d)$ and have a characterization of all matrices of $H(d)$.

As these dimensions are prime powers, by Proposition~\ref{prop:prime-power-dim} in these cases there exist MUBs of size $d + 1$, $\{ \bm{U}_1, \dots, \bm{U}_{d+1} \}$, for which moreover there are explicit constructions given in \cite{combescure2007circulant,combescure2009block,10.1007/978-3-540-24633-6_10}.
A direct construction of $\mu$ is then
\[ 
    \mu_0 \colonequals \frac{1}{d + 1}\sum_{i \neq j} \delta_{\bm{U}_i^*\bm{U}_j},
\]
which satisfies
\[
    \delta_e + \mu_0 = \frac{1}{d + 1}\sum_{i = 1}^{d + 1} \delta_{\bU_i^* \bU_i} + \frac{1}{d + 1}\sum_{i \neq j} \delta_{\bm{U}_i^*\bm{U}_j} = \frac{1}{d + 1} \sum_{i, j = 1}^{d + 1}\delta_{\bm{U}_i^*\bm{U}_j} \succeq (d + 1) \nu,
\]
the last claim holding by checking the value of the Fourier transform of either side on the trivial representation, in which case both sides equal 1, and on any other irrep $\pi$, in which case $\what{\nu}(\pi) = 0$ while the Fourier transform of the left-hand side is
\begin{align*}
\frac{1}{d + 1} \sum_{i, j = 1}^{d + 1}\what{\delta_{\bm{U}_i^*\bm{U}_j}}(\pi) 
&= \frac{1}{d + 1} \sum_{i, j = 1}^{d + 1}\pi(\bm U_i^*\bm U_j) \\
&= \frac{1}{d + 1}\left(\sum_{i = 1}^{d + 1} \pi(\bm U_i)\right)^*\left(\sum_{i = 1}^{d + 1} \pi(\bm U_i)\right) \\
&\succeq \bm 0,
\end{align*}
where, in the notation of Appendix~\ref{app:rep-Ud}, we have shown that $\delta_e+ \mu_0$ is of the form $\lambda^* * \lambda$.
Thus, $\mu_0$ gives a dual bound of $d + 1$ and is optimal for the dual program in these dimensions.

In fact, as we show in Appendix~\ref{sec: symmetries}, replacing $\delta_{\bH}$ with $\mu_{[\bH]}$ in such a construction does not affect feasibility or the dual value, so the same holds for the ``smoothed'' measure
\[
    \mu \colonequals \frac{1}{d + 1} \sum_{i \neq j} \mu_{[\bm{U}_i^*\bm{U}_j]}.
\]
We then consider what the equivalence classes appearing in the summation are.
In dimension $d = 2,3,5$ there is only one equivalence class of Hadamard matrices, equal to $[\bm F_d]$, so we have $\mu = d \mu_{[\bm F_d]}$.
In dimension $d = 4$ not all Hadamard matrices are equivalent. Instead there is a single continuous family of Hadamard matrices that belong to several different equivalence classes.
However, we do have an explicit construction \cite{brierley2010mutually} of $\bU_1, \dots, \bU_5$ a maximal collection of MUBs in $\CC^4$, and computing the pairwise differences we find that $[\bm{U}_i^*\bm{U}_j] = [\bm{F}_2 \otimes \bm{F}_2]$ for all $i \neq j$ (this is also proved in \cite{matolcsi2013systems}). So, when $d = 4$ the above sum still collapses and we have $\mu = 4\mu_{[\bm{F}_2 \otimes \bm{F}_2]}$.

\subsection{General construction}

It remains unclear if the above pattern will repeat in higher dimensions. There are several variants of constructions of MUBs in prime power dimensions \cite{combescure2007circulant,combescure2009block,10.1007/978-3-540-24633-6_10}, and it as an interesting problem to determine whether the pairwise differences of the associated unitary matrices are all equivalent to a suitable tensor product of Fourier matrices for all of these constructions.

Nonetheless, the above observations make it natural to predict the following general dual certificate construction.
The following directly implies Conjecture~\ref{conj: strong lower bound}.
\begin{Conj}\label{conj: strong dual measure}
    For any $d$, there is a Hadamard matrix $\bm H \in U(d)$ such that
    \begin{equation} \label{eq:hadamard-positivity}
        \delta_e +d \mu_{[\bm H]} \succeq (d+1)\nu,
    \end{equation}
    or, equivalently, so that for all $\pi \in \sI(U(d))$,
    \begin{equation}
        \what{\mu_{[\bm H]}}(\pi) \succeq -\frac{1}{d}\bm I.
    \end{equation}
    Moreover, if $d = p_1^{k_1} \cdots p_m^{k_m}$ is the prime factorization of $d$, then it is possible to take
    \[ \bm H = \bigotimes_{k = 1}^m \bm F_{p_i}^{\otimes k_i}. \]
\end{Conj}
\noindent
We have seen above that this is true for $2 \leq d \leq 5$.
In the remainder of the paper we will give partial results towards this conjecture for $d = 6$.
In that case, $\bm{F}_6, \bm{F}_2 \otimes \bm{F}_3 \text{ and } \bm{F}_3 \otimes \bm{F}_2$ are all equivalent Hadamard matrices, so $\mu_{[\bm{F}_6]} = \mu_{[ \bm{F}_2 \otimes \bm{F}_3]} = \mu_{[ \bm{F}_3 \otimes \bm{F}_2]}$ and we may simply take $\bH = \bF_6$.
This is true more generally for $d$ a squarefree number (a product of distinct primes), but not otherwise (as may be checked for the case $d = 4$).

\section{Dual bound in dimension \texorpdfstring{$d=6$}{d = 6}: Proof of Theorem~\ref{th:lower bound}}\label{sec: d=6}

In this section, we focus on the choice $\bH = \bF_6$.
By Lemma~\ref{lem:dual-mub-finite}, it suffices to show that $\what{\mu_{[\bH]}}(\pi) \succeq -\frac{1}{6}\bm I$ for all irreps $\pi$ of $U(6)$ that appear in decompositions of $\rho^{\otimes r} \otimes (\rho^*)^{\otimes r^{\prime}}$ for any $r, r^{\prime} \geq 0$ with $r + r^{\prime} \leq 6$, where $\rho$ is the natural representation of $U(d)$.

Thus we will first need to identify this collection of irreps, and then develop tools for showing the necessary positivity condition.
We give deeper background on the representation theory of $U(d)$ in Appendix~\ref{app:rep-Ud}, but here we give a brief overview that suffices to specify our calculations.

\subsection{Overview of necessary representation theory of \texorpdfstring{$U(d)$}{U(d)}}

$\sI(U(d))$ is parametrized by tuples $\bw \in \ZZ^d$ whose entries are decreasing, which we denote $\bw \in \ZZ^d_{\downarrow}$.
We refer to $\bw$ as the \emph{weight} of the associated $\pi$.
We also write $\lvert \bw \rvert \colonequals \sum_{i = 1}^d \lvert w_i \rvert$.

There are two basic representations within whose tensor products we may identify instances of each $\pi$.
The first is the representation $\rho$ discussed above and in Lemma~\ref{lem:dual-mub-finite}.
The second is a special non-trivial one-dimensional representation $\eta$ with $\eta(\bU) = \overline{\det(\bU)} = \det(\bU)^{-1}$.

We may produce a concrete instance of $\pi$ indexed by $\bw$ as follows.
Let $f_i = w_i - w_d$, this $\bf$ forming a partition of $n = n(\bw) \colonequals \sum_{i = 1}^d (w_i - w_d)$.
Then, $\pi$ will be a subrepresentation of $\eta^{\otimes \lvert w_d \rvert} \otimes \rho^{\otimes n}$, on which $\bU \in U(d)$ acts as the matrix $\det(\bU)^{-\lvert w_d \rvert} \bU^{\otimes n}$.
Specifically, $\pi$ may be identified as the subrepresentation on the image of the \emph{Young symmetrizer} associated to any \emph{Young tableau} whose \emph{shape} is the partition $\bbf$.
We define these notions in greater detail in Appendix~\ref{app:rep-Ud}, but, in short, the Young symmetrizer is a linear operator $\bY_{\bbf}: (\CC^d)^{\otimes n} \to (\CC^d)^{\otimes n}$ with a combinatorial definition involving the signs of certain permutations associated to $\bbf$.
With a particular default choice of Young tableau associated to $\bbf$ as detailed in Appendix~\ref{app:rep-Ud}, we let $V_{\bw} \subset (\CC^d)^{\otimes n}$ denote this subspace, which is then a realization of $\pi$ indexed by $\bw$.
We write $\pi_{\bw}$ for this representation when we wish to not be particular about this specific realization.

Finally, as we show in Corollary~\ref{weights pol degree}, the decomposition of $\rho^{\otimes r} \otimes (\rho^*)^{\otimes r^{\prime}}$ into irreps can involve only those $\pi_{\bw}$ with $\lvert \bw \rvert \leq r + r^{\prime}$, of which there are finitely many.

\subsection{Initial steps}

Suppose $\pi \in \sI(U(d))$.
Expanding the definition,
\begin{align*}
    \widehat{\mu_{[\bH]}}(\pi) 
    &= \int \pi(\bU) d\mu_{[\bH]}(\bU) \\
    &= \frac{1}{(d!)^2}\sum_{\sigma \in S_d} \sum_{\sigma' \in S_d} \int_{\TT^d}\int_{\TT^d}\pi( \bm P_\sigma \bm D_{\bm\lambda} \bm H  \bm D_{\bm\lambda^{\prime}} \bm P_{\sigma'})d\bm\lambda d\bm\lambda^{\prime} \\
    &= \bm \Pi_{S_d} \bm \Pi_{\TT^d} \pi(\bH) \bm \Pi_{\TT^d} \bm \Pi_{S_d},
\end{align*}
where we set
\begin{align*}
    \bm \Pi_{S_d} &\colonequals \frac{1}{d!}\sum_{\sigma \in S_d} \pi(\bm P_{\sigma}), \\
    \bm \Pi_{\TT^d} &\colonequals \int_{\TT^d} \pi(\bD_{\bm\lambda})d\bm\lambda = \frac{1}{(2\pi)^d} \int_0^{2\pi}\dots \int_0^{2\pi} \pi(\bm D_{e^{i t_1}, \dots, e^{i t_d}})\, dt_1 \cdots d t_d.
\end{align*}
These are projections to the subspaces of $\CC^{\dim \pi}$ invariant under the actions of the subgroups $S_d$ and $\TT^d$ of $U(d)$ (with the former identified as the subgroup of permutation matrices).
In fact, we have 
\begin{align*}
    \bm\Pi_{\mathbb{T}^d \rtimes S_d} \colonequals \bm \Pi_{S_d} \bm{\Pi}_{\mathbb{T}^d} = \bm{\Pi}_{\mathbb{T}^d} \bm \Pi_{S_d},
\end{align*}
as may be checked for instance by recognizing that either product involves an integration with respect to a measure over $\TT^d \rtimes S_d$ that satisfies the properties of a Haar measure.
Thus this product is likewise a projection to the subspace invariant under the group of generalized permutations $\TT^d \rtimes S_d$ (see Definition~\ref{def:gen-perm}).
In effect, we need only work over this subspace to verify the necessary positivity condition for $\pi$, so it will play a crucial role in our analysis.
We thus define:
\begin{equation}
    \widetilde{V}_{\pi} \colonequals \mathrm{img}(\bm\Pi_{\mathbb{T}^d \rtimes S_d}).
\end{equation}

Our plan is then as follows.
First, we will develop some machinery for computing $\dim(\widetilde{V}_{\pi})$, and we will see that for the particular subspace of polynomials we are interested in, we will only need to consider $\pi$ for which $\dim(\widetilde{V}_{\pi}) \leq 1$.
When $\dim(\widetilde{V}_{\pi}) = 0$, then $\what{\mu_{[\bH]}}(\pi) = 0$ and the condition we need holds immediately.
The only non-trivial case will then be when $\dim(\widetilde{V}_{\pi}) = 1$.
In this case, for some $\bv$ we can write $\bm\Pi_{\TT^d \rtimes S_d} = \bv\bv^{\top}$, and the only non-zero eigenvalue of $\what{\mu_{[\bH]}}(\pi)$ is $\bv^{\top} \pi(\bH) \bv$, so it suffices to perform the \emph{scalar} computation of this value.

\subsection{Projection to and interpretation of \texorpdfstring{$\widetilde{V}_{\pi}$}{tildeV\_pi}}

We now proceed to a concrete description of $\widetilde{V}_{\pi}$, which we will see is actually a natural object from the perspective of the representation theory of $U(d)$ and even one that has been studied occasionally in the past.
Let us write $\widetilde{V}_{\bm w}$ for $\widetilde{V}_{\pi}$, where $\pi$ is indexed by $\bw \in \ZZ^d_{\downarrow}$.
Per the above discussion, the subspace $\widetilde{V}_{\bw} \subset V_{\bw}$ is that which is fixed under the actions of the subgroups $\TT^d$ and $S_d$ of $U(d)$.
As we have seen above, the projections to the respective fixed subspaces of either group commute, so we may apply them in either order to identify $\widetilde{V}_{\bw}$.
Thus let us consider these projections one at a time below.

\subsubsection{Subspace invariant under \texorpdfstring{$\mathbb{T}^d$}{T\^d}}
\label{sec:subspace-Td}

We first consider when $\bm \Pi_{\TT^d} \neq \bm 0$.
Computing the trace, we find
\begin{align*}
    \tr(\bm\Pi_{\mathbb{T}^d}) 
    &= \tr\left(\int_{\mathbb{T}^d} \pi(\bm D_{\bm\lambda})d\bm \lambda\right) \\
    &= \int_{\mathbb{T}^d} \tr(\pi(\bm D_{\bm\lambda}))d\bm \lambda \\
    &= \int_{\TT^d} \left(\sum_{i = 1}^d \lambda_i\right)^n \prod_{i = 1}^d \lambda_i^{-\lvert w_d \rvert} d\bm\lambda.
\end{align*}
Visibly, upon expanding the power of $\sum_{i = 1}^d \lambda_i$, each term that is not a constant will integrate to zero, and thus the above will be non-zero if and only if $n = d\lvert w_d \rvert$, which is equivalent to $\sum_{i = 1}^d w_i = 0$.
We then find the following substantial restriction.
\begin{Lemma}\label{inv diag}
If $\bm\Pi_{\mathbb{T}^d} \neq \bm 0$, then the weight $\bm{w}$ of $\pi$ satisfies $\sum_{i=1}^d w_i = 0$.
\end{Lemma}
\noindent
We detail in Appendix~\ref{app:rep-Ud} how, equipped with a description of a basis of $V_{\bm w}$ in terms of Young tableaux, we may also identify a basis of $\bm \Pi_{\TT^d}(V_{\bm w})$ described in a similar way.

Referring for a moment back to our original task, we see that in fact we only need to consider the representations $\rho^{\otimes r} \otimes (\rho^*)^{\otimes r}$, i.e., the case $r = r^{\prime}$.
For the case of studying $B_{\leq 6}(6)$, we then be interested in $0 \leq r \leq 3$.
Moreover, among the $\pi_{\bw}$ occurring in the decompositions of these representations, we will only be interested in those with $\sum w_i = 0$ and $d\lvert w_d \rvert = r$.

Let us comment briefly on the role that subspaces like $\bm\Pi_{\TT^d}(V_{\bm w})$ play in the representation theory of Lie groups more generally. 
$\TT^d$ is a \emph{torus subgroup} of $U(d)$, so called because it is, as a group, isomorphic to a product of unit circles in $\CC$, which is, as a topological space, homeomorphic to a torus. 
Those torus subgroups that are \emph{maximal}, as $\TT^d < U(d)$ is, play a special role in the classification of representations of Lie groups. 
In this classification, one argues as follows: for $\pi$ irreducible as a representation of $U(d)$, $\pi$ is also a representation of $\TT^d$, though not necessarily an irreducible one.
Thus one may consider the decomposition of $\pi$ into irreducible subrepresentations of $\TT^d$, which is straightforward since $\TT^d$ is abelian: all such subrepresentations are one-dimensional, and $\TT^d$ acts with $\bm\lambda$ acting as $\prod_{i = 1}^d \lambda_i^{a_i}$, for some tuple $\bm a = (a_1, \dots, a_d) \in \ZZ^d$.
So, to each $\pi$ is associated a collection of \emph{weights} $\bm a$ and associated subspaces on which $\TT^d$ acts with those weights.
It turns out that the $\bw$ associated to $\pi$ is, in a certain ordering, the \emph{highest weight} appearing in $\pi$, and this highest weight suffices to determine $\pi$ completely.
On the other hand, $\widetilde{V}_{\pi}$ is the subspace associated to the weight $\ba = \bm 0$, and for this reason is called the \emph{zero-weight subspace} of $V_{\pi}$.
See Appendix~\ref{app:irreps} for further technical details.
Much the same story applies to all compact Lie groups with respect to a maximal torus subgroup; see \cite{fulton2013representation} for a more general presentation (note that such presentations, however, usually treat $SU(d)$ instead of $U(d)$, as the latter is not semisimple once $d \geq 2$).

\subsubsection{Subspace invariant under \texorpdfstring{$S_d$}{S\_d}}\label{inv perm}

We are next interested in the further subspace of the zero-weight subspace that is fixed by the action of $S_d$.
Here again it is useful to recognize the more general context in Lie theory.
Here it is more convenient to work over $SU(d) < U(d)$.
The above discussion applies just as well to $SU(d)$, only with maximal torus isomorphic to $\TT^{d - 1}$ due to the determinant constraint (concretely, one may view the torus as diagonal matrices with the last diagonal entry determined by the first $d - 1$).
The group $S_d < SU(d)$ is the \emph{Weyl group} of $SU(d)$, and is formed as $N(\TT^{d - 1}) / \TT^{d - 1}$, where $N(\TT^{d - 1}) = \{\bU \in SU(d): \bU\bD\bU^{-1} \in \TT^{d - 1} \text{ for all } \bD \in \TT^{d - 1}\}$ is the \emph{normalizer} of $\TT^{d - 1}$ in $SU(d)$.

From this definition, we see that it is a general phenomenon that the Weyl group has a well-defined action on the zero-weight subspace of any irrep, i.e., the zero-weight subspace gives a representation of the Weyl group.
The question of characterizing this representation for various Lie groups has received some attention in the literature \cite{Gutkin-1973-RepresentationsWeylGroup,kostant1976macdonald,gay1976characters,reeder1998zero}, and we will draw on one of these results in our calculations.
In particular, the following result gives a recipe for computing the character of this representation of the Weyl group, for the special case of $SU(d)$ with Weyl group $S_d$.
We will use this result solely for symbolic computations, so the reader need not understand the details of the statement---all that matters is that it is, in principle, possible to compute the character of the representation we have described above.
\begin{Prop}[Theorem 2 of \cite{gay1976characters}]
    \label{weight zero character}
    Write $s \colonequals \lvert w_d \rvert$.
    Let $n = sd$ as above, and let $H \colonequals S_s \times \cdots S_s < S_n$.
    This may be viewed as the stabilizer of the element
    \[ \underbrace{e_1 \otimes \cdots \otimes e_1}_{s \text{ times}} \otimes \cdots \otimes \underbrace{e_d \otimes \cdots \otimes e_d}_{s \text{ times}} \]
    when $S_n$ acts on $(\CC^d)^{\otimes n}$ by permuting the axes.
    Let $N(H) < S_n$ be the normalizer of $H$, so that $N(H) / H$ is isomorphic to $S_d$.
    For each character $\psi$ of $S_d$, let $\widehat{\psi}$ be the character of $N(H)$ that is formed by extending $\psi$ to be constant on cosets under the above quotient.
    Then, let $\widehat{\psi}^{S_n}$ be the induced character on $S_n$.
    Suppose that there are coefficients $c_{\phi, \psi}$ for each $\phi$ a character of $S_n$ such that this character admits the expansion
    \[ \widehat{\psi}^{S_n} = \sum_{\phi \text{ a character of } S_n} c_{\phi, \psi} \phi. \]
    Let $\chi_0$ be the character of $S_d$ acting on $\Pi_{\mathbb{T}^d}(V_{\bm{w}})$ by permuting the indices of the standard basis vectors.
    Let $\phi_0$ be the character of the representation of $S_n$ indexed by the partition $(w_1 - w_d, \dots, w_{d - 1} - w_d)$.
    Then, $\chi_0$ admits the expansion
    \[ \chi_0 = \sum_{\psi \text{ a character of } S_d} c_{\phi_0, \psi} \psi. \]
\end{Prop}

We emphasize one important detail, which stems from \cite{gay1976characters} working over $SU(d)$ while we work over $U(d)$.
The above describes the character of $S_d$ acting on $\bm \Pi_{\TT^d}(V_{\bm w})$ where $\sigma$ acts as $\bm P_{\sigma}^{\otimes n}$.
This is \emph{not} the action associated to the representation $V_{\bm w}$.
That action involves an extra determinant term, with $\sigma$ acting as $\pi(\bm P_{\sigma}) = \det(\bm P_{\sigma})^{\lvert w_d \rvert} \bm P_{\sigma}^{\otimes n} = \sgn(\sigma)^{\lvert w_d \rvert}\bm P_{\sigma}^{\otimes n}$.
Taking into account this adjustment, we finally obtain the following means of computing $\dim(\widetilde{V}_{\bw})$.

\begin{Cor}
    \label{invariant subspace dimension}
    Let $\chi$ be the character of the trivial representation of $S_d$ if $\lvert w_d \rvert$ is even, and the character of the sign representation if $\lvert w_d \rvert$ is odd.
    Then, in the setting of Proposition~\ref{weight zero character}, $\dim(\widetilde{V}_{\bm w}) = \langle \chi_0, \chi \rangle = c_{\phi_0, \chi}$.
\end{Cor}
\noindent
With this fact in hand,  $\dim(\widetilde{V}_{\bm w})$ may be computed using standard symbolic algebra tools; we have used the \texttt{SageMath} package for this purpose.
What will be convenient is that often $\dim(\widetilde{V}_{\bm w}) = 0$, and for small $\|\bm w \|_1$ when the dimension is non-zero then it is 1.
We list all indices $\bm w$ where $\lvert \bw \rvert \leq 10$ and where $\dim(\widetilde{V}_{\bm w}) > 0$ in Table~\ref{tab:small-w-dimensions}; we will only look at $\lvert \bm w \rvert \leq 6$ in the sequel, but include further results for the sake of completeness.

\begin{remark}
    It seems plausible based on the data presented in Table~\ref{tab:small-w-dimensions} to conjecture that $\dim(\widetilde{V}_{\bm w}) \neq 0$ only when $\bm w$ has at most two non-zero positive entries and two non-zero negative entries.
    It is an interesting open problem to derive such a constraint from Corollary~\ref{invariant subspace dimension}.
\end{remark}

\begin{table}[t]
    \centering
    \begin{tabular}{c|rrrrrr|c}
        $r = \lvert \bm w \rvert / 2$ & \multicolumn{6}{c|}{$\bm w$} & $\dim(\widetilde{V}_{\bm w})$ \\
        \hline
        0 & 0 & 0 & 0 & 0 & 0 & 0 & 1 \\
                \hline
        2 & 2 & 0 & 0 & 0 & 0 & $-2$ & 1 \\
        \hline
        3 & 3 & 0 & 0 & 0 & 0 & $-3$ & 1 \\
        \hline
        \multirow{5}{*}{4} & 2 & 2 & 0 & 0 & $-2$ & $-2$ & 1 \\
        & 2 & 2 & 0 & 0 & 0 & $-4$ & 1 \\
        & 3 & 1 & 0 & 0 & $-1$ & $-3$ & 1 \\
        & 4 & 0 & 0 & 0 & $-2$ & $-2$ & 1 \\
        & 4 & 0 & 0 & 0 & 0 & $-4$ & 2 \\
        \hline
        \multirow{9}{*}{5} & 3 & 2 & 0 & 0 & $-2$ & $-3$ & 1 \\
        & 3 & 2 & 0 & 0 & $-1$ & $-4$ & 1 \\
        & 3 & 2 & 0 & 0 & 0 & $-5$ & 1 \\
        & 4 & 1 & 0 & 0 & $-2$ & $-3$ & 1 \\
        & 4 & 1 & 0 & 0 & $-1$ & $-4$ & 1 \\
        & 4 & 1 & 0 & 0 & 0 & $-5$ & 1 \\
        & 5 & 0 & 0 & 0 & $-2$ & $-3$ & 1 \\
        & 5 & 0 & 0 & 0 & $-1$ & $-4$ & 1 \\
        & 5 & 0 & 0 & 0 & 0 & $-5$ & 2 \\
        \hline
    \end{tabular}
    
    \vspace{1em}
    
    \caption{All cases of $\bm w \in \ZZ^6_{\downarrow}$ where $0 \leq \lvert \bm w \rvert \leq 10$ and where $\dim(\widetilde{V}_{\bm w}) > 0$. We note that there are no such cases with $\lvert \bm w \rvert = 2$.}
    \label{tab:small-w-dimensions}
\end{table}

\subsection{Final computer verification}
\label{sec:final-verification}

Finally, we describe the last part of our proof, which is essentially a brute-force computer-assisted treatment of the few cases where $\dim(\widetilde{V}_{\bw}) = 1$ for $\lvert \bw \rvert \leq 6$.
In particular, per the results presented in Table~\ref{tab:small-w-dimensions}, it suffices to consider $\bw = (k, 0, 0, 0, 0, -k)$ for $k \in \{2, 3\}$.
The code for performing these verifications, which is implemented in the \texttt{SageMath} system, is available in the online supplementary materials \cite{supplementary}.
Here we give a brief overview of how these calculations are implemented; further details of optimizations made in the code are also discussed in Appendix~\ref{app:computational}.

We would like to show that
\[ -\frac{1}{6}\bm I \preceq \widehat{\mu_{[\bm H]}}(\pi) = \bm \Pi_{\TT^d \rtimes S_d} \pi(\bm H) \bm \Pi_{\TT^d \rtimes S_d}. \]
The preceding discussion has shown that, for the two cases we are interested in, $\dim(\widetilde{V}_{\bw}) = 1$, so it suffices to produce some $\bm 0 \neq \bv \in \widetilde{V}_{\bw}$ (all such vectors will be scalar multiples of one another) and verify that
\begin{equation}
    \label{eq:H-ineq}
    -\frac{1}{6} \|\bv\|^2 \leq \bv^{\top} \pi(\bm H) \bv.
\end{equation}
We describe how to find such non-zero $\bv$ in Appendix~\ref{app:tV-nonzero} using the combinatorial description of a basis of $\bm \Pi_{\TT^d}(V_{\bm w})$ and the formula for $\bm \Pi_{S_d}$.
It is important that this $\bv$ is very sparse, since it belongs to $(\CC^d)^{\otimes n}$ for $n = d\lvert w_d \rvert$, which can have dimension as large as $6^{6 \cdot 3} > 10^{14}$ for $\bw = (-3, 0, 0, 0, 0, 3)$.
As shown there, $\bv$ will lie in the subspace spanned by the 
\[
    \bm x_{\bm i} \colonequals \sum_{\sigma \in S_d}\sgn(\sigma)^{\lvert w_d \rvert} \bm e_{\sigma(\bm{i})} = \sum_{\sigma \in S_d}\sgn(\sigma)^{\lvert w_d \rvert}  \bm e_{\sigma(i_1)} \otimes \bm e_{\sigma(i_2)} \otimes \dots \otimes \bm e_{\sigma(i_{d\lvert w_d \rvert})},
\]
for $\bm i \in [d]^{n}$ where every index appears exactly $\lvert w_d \rvert$ times in $\bm i$, and will have all entries belonging to $\{-1, 0, 1\}$.
We then compute $\|\bv\|^2$ directly by counting the non-zero entries in $\bv$.

In this section we describe the structure of $\pi(\bm H)$ sufficiently to specify the rest of this computation.
As we view $\pi$ as acting on a subspace of $(\CC^d)^{\otimes n}$, we may identify
\[ \pi(\bH) = \det(\bH)^{-\lvert w_d \rvert} \bH^{\otimes n}. \]
The determinant term is not trivial to compute, but fortunately was treated in previous literature.
\begin{Prop}[Table 1 of \cite{dickinson1982eigenvectors}]
    Suppose $d = 4q + r$ for some $r \in \{0, 1, 2, 3\}$. Then,
    \begin{equation}
        \det(\bm F_d) = \left\{\begin{array}{ll} 
        (-1)^q i & \text{if } r = 0, \\
        (-1)^q & \text{if } r = 1, \\
        (-1)^{q + 1} & \text{if } r = 2, \\
        (-1)^{q + 1}i & \text{if } r = 3.
        \end{array}\right.
    \end{equation}
\end{Prop}
In particular, in our case of $d = 6$ we have $q = 1$ and $r = 2$, so $\det(\bm F_6) = 1$ and we may discard this term.
Writing $s \colonequals \lvert w_d \rvert$, the right-hand side of \eqref{eq:H-ineq} will then be a linear combination of quantities of the form
\begin{align*}
 f(\bm i, \bm j) 
 &\colonequals \bx_{\bm i}^{\top} \pi(\bm H) \bx_{\bm j} \\
 &= \det(\bm F_d)^{-s}\frac{1}{(d!)^2d^{\frac{ds}{2}}} \sum_{\sigma, \sigma' \in S_d}  (\sgn(\sigma  \sigma' ) )^{s} \exp\left(\frac{2\pi i}{d} \langle \sigma(\bm{i}),\sigma'(\bm{j}) \rangle\right).
 \end{align*}

One can see that some of the information contained in a pair of indices $\bm{i},\bm{j}$ is redundant, and, in order to compute $f(\bm i, \bm j)$, we only need to keep track of how often various indices occur in the same position in $\bm{i}$ and $\bm{j}$. For this we introduce the formalism of the \emph{$\bm{S}$-matrix}.

\begin{Def}[$\bm{S}$-matrix]
    Given a pair of indices $\bm{i,j}$, we define $\bm{S^{i,j}} \in (\mathbb{Z} / d\ZZ)^{d \times d} $
    by 
    \[
        S_{m,n}^{\bm i, \bm j} = \lvert \{t : i_t = m \text{ and } j_t = n  \} \rvert \mod d.
    \]
    This matrix satisfies $\sum_k S_{k,\ell}^{\bm i, \bm j} = \sum_{\ell} S_{k,\ell}^{\bm i, \bm j} = s \mod d$ for every $j$, which is sometimes called the property of being a \emph{magic square} (modulo $d$).
\end{Def}
\noindent
Equivalently, such $\bS$ is the adjacency matrix of an $s$-regular bipartite graph with $d$ nodes on either side of the bipartition and repeated edges permitted, where furthermore the counts of repeated edges may be interpreted modulo $d$.

These matrices determine the function we are interested in as follows.
\begin{Lemma}
    $f(\bm i, \bm j)$ is determined by the corresponding $\bm S = \bm{S^{i,j}}$ matrix, in particular 
    \[ f(\bm i, \bm j) = f(\bm S) = \frac{\det ( \bm F_d)^{-s}}{ (d!)^2 d^{\frac{ds}{2}}} \sum_{\sigma,\sigma' \in S_d} \sgn(\sigma \sigma')^s \exp\left(\frac{2 \pi i}{d}\sigma(\bm c)^\top \bm{S} \sigma^{\prime}(\bm c)\right) \]
    for $\bm c = (1,2 \dots, d)^\top$.
\end{Lemma}
When $\bm S$ is interpreted as the adjacency matrix of an $s$-regular bipartite multigraph with $d$ vertices on each side of the bipartition, then the computation above may be viewed as indexed by all labellings of either side with the numbers $1, \dots, d$, and where the inner product in the exponential is the sum of products of pairs of numbers on opposite sides of each edge.

We complete our symbolic verification by computing $f(\bm S)$ for \emph{all} $\bm S$ with row and column sums equal to 2 and 3 as a preprocessing step, and then computing the requisite summation of those quantities given by the Young symmetrizer.
We provide this intermediate result together with the code in \cite{supplementary}; indeed, understanding what governs the behavior of $f(\bm S)$ is likely to be one of the main technical obstacles to a more conceptual and general proof of positivity.

We remark on just one final computational shortcut: as is clear from the formulas above, $f(\bm S)$ is invariant under permutations of the rows and columns of $\bm S$, except for incurring a change in sign according to the sign of the permutation if $s$ is odd.
Thus it suffices to precompute the values of such $\bm S$ up to permutations.
By repeatedly applying Hall's marriage theorem, we find that any $\bm S$ is a sum of permutation matrices, $\bm S = \bm P_1 + \bm P_2 + \cdots + \bm P_s$.
We may then multiply $\bm S$ on the left by $\bm Q \bm P_1^{-1}$ and on the right by $\bm Q^{-1}$, obtaining
\[ \bm Q \bm P_1^{-1} \bm S \bm Q^{-1} = \bm I + \sum_{k = 2}^s \bm Q \bm P_1^{-1} \bm P_k \bm Q^{-1}. \]
Thus, it suffices to precompute $f(\bm S)$ for those cases where $\bm P_1 = \bm I$ and $\bm P_2$ is some canonical representative of its conjugacy class, which depends only on the cycle type.
For $s = 1$ this leaves only one matrix, for $s = 2$ and $d = 6$ it leaves 11 matrices (as there are 11 partitions of the number 6), and for $s = 3$ and $d = 6$ it leaves $11 \cdot 6! = 7920$ partitions.
Of course, these numbers quickly grow for increasing $s$, but for $s \leq 3$ and $d = 6$ these computations fortunately remain tractable.

Further details on how these results are substituted into the main computations of $\|\bv\|^2$ and $\bv^{\top}\pi(\bH)\bv$ are given in Appendix~\ref{app:computational}.

\begin{remark}
    We do not know, in general, how to compute $f(\bm i, \bm j)$ in closed form, and this is an intriguing open question that would likely need to be addressed to generalize our results to higher polynomial degrees.
    A more principled approach might take advantage of the representation theory of the action of $S_n$ on ``balanced partitions'' of $[n] = [\lvert w_d \rvert \cdot d]$ into $d$ parts of size $\lvert w_d \rvert$.
    Unfortunately, the decomposition of the associated representation into irreducible representations of $S_n$ appears to be unknown in general; see, e.g., Chapter 12 of \cite{GM-2016-ErdosKoRadoAlgebraic}.
    The recent paper \cite{GP-2021-MUBOptimizationSymmetry}, which explores a different approach using non-commutative sum-of-squares optimization to bound the sizes of MUBs, uses this symmetry group as well, albeit only for numerical computations.
\end{remark}

\addcontentsline{toc}{section}{References}
\bibliography{refs.bib}

\appendix

\section{Sufficiency of symmetric dual certificates}\label{sec: symmetries}

Recalling the original problem of MUBs, we notice that their definition is invariant under permutations of the basis vectors and multiplications of each basis by a complex number of unit norm.
It is therefore reasonable to expect that the upper bound obtained from \eqref{eq: mub mini} or the dual bound from \eqref{eq: mub maxi} should also be invariant under such transformations. We formalize this intuition in the following two theorems, which can be applied by taking $W = \mathbb{T}^d \rtimes S_d$ (per Definition~\ref{def: equiv had}).

\begin{Th}\label{inv function}
    Suppose that $G$ is a compact group and $W \leq G$ is a closed subgroup such that
    \[
        wA = Aw = A 
    \]
    for all $w \in W$. Then, we may add to \eqref{eq: mub mini} the constraint that $h$ invariant under $W$, i.e., that $h(wg) = h(gw) = h(g)$ for any $g \in G$ and $w \in W$, without changing the value of $B(G, A)$.
\end{Th}

\begin{proof}
    Since $W$ is a closed subgroup and $G$ is compact, $W$ is also a compact group and so is endowed with a left- and right-invariant Haar probability measure $\nu_W$.
    We consider replacing some $h$ feasible for \eqref{eq: mub mini} by
    \[
        h_W(g) \colonequals \int_W  \int_W  h(w_1 g w_2) d\nu_W(w_1) d\nu_W(w_2).
    \]
    Clearly $h_W$ is still in $C_{\leq 0, A^c}$.
    Since $h$ is positive definite, we may write it as
    \[
        h(g) =\langle \pi(g) \bv, \bv \rangle 
    \]
    for some $\pi: G \to \sH$ a unitary representation for some Hilbert space $\sH$ and $\bv \in \sH$.
    We then we obtain 
    \begin{align*}
        h_W(g) &= \left\langle \pi(g) \int_W \pi(w) d \nu_W(w)\bv, \int_W \pi(w^{-1}) d \nu_W(w)\bv \right\rangle \\
         &= \left\langle \pi(g) \int_W \pi(w) d \nu_W(w)\bv, \int_W \pi(w) d \nu_W(w)\bv \right\rangle,
    \end{align*}
    where we used that the pushforward of $\nu_W$ under the map $w \mapsto w^{-1}$ is also a left- and right-invariant probability measure on $W$ and so must equal $\nu_W$.
    This shows that $h_W$ is also a positive definite function.
    Since $h(g) \leq h(e)$ for all $g \in G$, we have $h_W(e) \leq h(e)$ and thus $h_W(e)$ is feasible for \eqref{eq: mub mini}, has the required invariance, and has an objective value at most that of $h$.
\end{proof}

We note that the polynomial $h_0$ used in \cite{KMW-2018-PositiveDefiniteMUBs} to rederive the Welch bound (our \eqref{h0-poly}) already fulfills this invariance property.

To formulate an analogous reduction for the dual program, for $w \in G$ and $f \in C(G)$, define
\begin{align*}
    (L_wf)(g) &\colonequals f(w^{-1}g), \\
    (R_wf)(g) &\colonequals f(gw).
\end{align*}
These operators also admit adjoints acting on measures, which integrate $f \in C(G)$ through
\begin{align*}
    \int_G f(g) dL_w^* \mu &= \int_{G} f(w^{-1} g) d \mu(g), \\
    \int_G f(g) dR_w^* \mu &= \int_{G} f(gw) d\mu(g).
\end{align*}

\begin{Th}\label{lem: inv measure}
    With $W \leq G$ as in Theorem~\ref{inv function}, we may add to \eqref{eq: mub maxi} the constraint that $\mu$ is invariant under $W$, i.e., that $L_w^*\mu = R_w^*\mu = \mu$ for any $w \in W$, without changing the value of $B^*(G, A)$.
\end{Th} 

\begin{proof}
    Suppose $\mu$ is feasible for \eqref{eq: mub maxi}, i.e., that $\delta_e + \mu \succeq c\nu $.
    Then, we want to show that $[\mu]$ defined by 
    \[
        [\mu] \colonequals \int_W \int_W R_{w_1}^*L_{w_2}^*\mu d \nu_W(w_1)d \nu_W(w_2),
    \]
    for $\nu_W$ the Haar probability measure on W, fulfills 
    \[
        \delta_e + [\mu] \succeq c\nu.
    \]
    We will show this by checking the value of the Fourier transform on each irrep of $G$.
    
    For the trivial representation, the associated scalar inequality is $1 + \mu(G) \geq c$, which is the same as the feasibility condition for $\mu$ evaluated on the trivial representation, since $\int_G d\mu = \int_G d[\mu] = \mu(G)$.
    Suppose now that $\pi$ is non-trivial.
    Then, the feasibility condition for $\mu$ gives us
    \[
        \bm{I}+\int \pi(g) d \mu \succeq 0.
    \]
    We then have
     \begin{align*}
        \bm{I}+ \int_W \int_W \int_G \pi(w_1^{-1} g w_2) d\nu_W(w_1) d \mu(g) d \nu_W(w_w) & = \bm{I}  +\bm P\int \pi(g) d \mu \bm P.
        \end{align*}
     For 
     \[
        \bm P = \int_W \pi(w) d \nu_W(w) = \int_W \pi(w^{-1}) d \nu_W(w) = \bm P^*,
    \]
    which, as the name suggests, is a projection onto the subspace invariant under $\pi(w)$ for any $w \in W$.
    Since we have
    \[
        \bm{I}+\int \pi(g) d \mu \succeq 0,
    \]
    we also have
        \[
        \bm P\left(\bm{I}+\int \pi(g) d \mu \right)\bm P \succeq 0,
    \]
    and since $\bm P$ is a projection we have $\bm{I} \succeq \bm P = \bm P\bm{I}\bm P$, whereby
    \[
    \bm{I} + \bm P \int_G \pi(g) d \mu(g) \bm P \succeq \bm P\left(\bm{I}+\int \pi(g) d \mu\right)\bm P \succeq 0,
    \]
    completing the proof.
\end{proof}

\begin{Rem}
    $S_d$, $\mathbb{T}^d$ and their semidirect product can be realized as subgroups of $U(d)$. All of these groups are equipped with Haar measures, which we will denote by $\nu_{\mathbb{T}^d}, \nu_{S_d}$ and  $\nu_{\mathbb{T}^d \rtimes S_d}$ respectively. We can abuse notation and reuse this notation for the push-forward of these under the natural inclusion as measures on $U(d)$. We then have, in the notation of Appendix~\ref{app:star-algebra},
    \[
        \nu_{\mathbb{T}^d \rtimes S_d} = \nu_{S_d} * \nu_{\mathbb{T}^d} =  \nu_{\mathbb{T}^d}* \nu_{S_d}
    \]
    as well as 
    \[
        \nu_H = \nu_H^* = \nu_H^* * \nu_H,
    \]
    for each group $H \in \{ S_d, \mathbb{T}^d, \mathbb{T}^d \rtimes S_d, U(d) \}$.
    
    The operation $[ \cdot]$ described in Theorem~\ref{lem: inv measure} can then be defined concisely as
    \[
    [\lambda] = \nu_{\mathbb{T}^d \rtimes S_d} * \lambda *  \nu_{\mathbb{T}^d \rtimes S_d}.
    \]
    Further, $\bm \Pi_{H}$ is equal to $\widehat{\nu_H}(\pi)$ whenever it occurs in the main text.
\end{Rem}

\section{Banach \texorpdfstring{$*$-algebra}{*-algebra} interpretation of positive definite measures}
\label{app:star-algebra}

We mention an alternative viewpoint on positive definite measures that clarifies some matters raised in the main text.
If $G$ is a compact group, then $\mathcal{M}(G)$ is a Banach $*$-algebra under the convolution product $\lambda * \mu$ defined by
\[
    \int_G f(g) \, d (\lambda * \mu)(g) \colonequals \int_G \int_G f(gh) \, d  \lambda(g) d \mu(h)
\]
and the involution $\lambda \mapsto \lambda^*$ defined by 
\[
\int_G f(g) \, d \lambda^*(g) = \overline{\int_G \overline{f(g^{-1})}\, d \lambda(g)}.
\]
We note that $\lambda$ here is \emph{a priori} a complex-valued measure, rather than the real-valued measures that we deal with in the majority of the main text, so we cannot merely cancel the two conjugations.

One can then check that, for any representation $\pi$ of $G$, the map that evaluates the Fourier transform at $\pi$,
\[
    \lambda \mapsto \widehat{\lambda}(\pi),
\]
is a $*$-algebra morphism.
That is, we have
\begin{align*}
\widehat{\lambda * \mu}(\pi) &= \widehat{\lambda}(\pi) \widehat{\mu}(\pi), \\
    \widehat{\lambda^*}(\pi) &= (\widehat{\lambda}(\pi))^*, \\
\widehat{\lambda+ \mu}(\pi) &= \widehat{\lambda}(\pi)+\widehat{\mu}(\pi).
\end{align*}
Combining this with Theorem~\ref{thm:pos-def-measures} then gives that all measures of the form $\lambda *\lambda^*$ are positive definite, i.e., $\lambda * \lambda^* \in \mathcal{M}_{\succeq 0}(G)$.

\section{Representation theory of \texorpdfstring{$U(d)$}{U(d)}}\label{app:rep-Ud}

The unitary group $U(d)$ consists of all $d \times d$ complex matrices $\bm{U}$ which satisfy $\bm{U}^*\bm{U} = \bm{I}$. It is a compact topological group (indeed, also a Lie group) under the subspace topology of $\CC^{d \times d}$ and the operation of matrix multiplication. 
We review some standard facts about its representation theory below, drawing on the standard references \cite{Fulton-1997-YoungTableaux,fulton2013representation} for general theory and \cite{Sengupta2012} for aspects of the specific case of $U(d)$.

\subsection{Young diagrams, tableaux, and symmetrizers}
\label{app:young}

Given a partition of $n$, $\bbf = (f_1, \dots, f_k)$ with $\sum f_i = n$ and $f_1 \geq f_2 \geq \dots \geq f_k > 0$, we can associate to it a \textit{Young diagram}, which is a collection of boxes with $f_1$ boxes in the first row, $f_2$ in the second, and so on. 
For example, this is the diagram associated to the partition $(5,3,1)$:
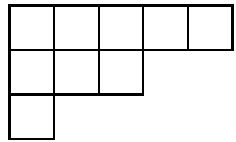
\begin{center}
\vspace{1em}
\ydiagram{ 5, 3, 1}
\vspace{1em}
\end{center}

A way of filling these boxes with numbers is a \textit{Young tableau} of \emph{shape} $\bm{f}$.
A Young tableau is called \emph{standard} if it is increasing in both the rows and columns, and \emph{semistandard} if it is non-decreasing in the rows and increasing in the columns. 
We call the \emph{default} standard tableau of a given shape (this is not typical terminology) the one obtained by putting in the numbers $1, \dots, n$ consecutively across the rows, as in, for example:
\begin{center}
\vspace{1em}
\ytableausetup{centertableaux}
    \begin{ytableau}
        1 & 2 & 3 & 4 & 5 \\
        6 & 7 & 8  \\
        9
    \end{ytableau}
\vspace{1em}
\end{center}

We will denote the set of semistandard Young tableaux with shape $\bm{f}$ and filled with numbers in $1, \dots ,m$ by  $\mathcal{Y}_{\bm{f}}(m)$. For instance,
\begin{equation}
\ytableausetup{centertableaux}
    \begin{ytableau}\label{ex tableau}
        1 & 1 & 1 & 1 & 3 \\
        2 & 2 & 2  \\
        3
    \end{ytableau}
\end{equation}
is in $\mathcal{Y}_{(5,3,1)}(3)$.

The \emph{content} $\bm{c}$ of a Young tableau $T \in \mathcal{Y}_{\bm{f}}(m)$ has $c_i$ equal to the number of times $i$ appears in $T$.
For example, the content of \eqref{ex tableau} is $(4,3,2)$.  $\mathcal{Y}_{\bm{f}}^{\bm{c}}(m)$ will denote the set of Young tableaux in  $\mathcal{Y}_{\bm{f}}(m)$ with content $\bm{c}$.

Finally, given a Young diagram of $n$ boxes associated to a partition $\bbf$ and $T_0$ the default standard Young tableau, we let $R \leq S_n$ be the subset of permutations that fix each row of $T_0$ and $C \leq S_n$ the subset of permutations that fix each column of $T_0$.
Suppose $V$ is some fixed vector space, then for $\sigma \in S_n$ we let $\bF_{\sigma}: V^{\otimes n} \to V^{\otimes n}$ be the linear operator permuting the tensor axes according to $\sigma$ (i.e., mapping $v_1 \otimes \cdots \otimes v_n \mapsto v_{\sigma(1)} \otimes \cdots \otimes v_{\sigma(n)}$ and extending this action by linearity to all tensors).
Clearly we have $\bF_{\sigma\tau} = \bF_{\sigma}\bF_{\tau}$ (indeed, the mapping $\sigma \mapsto \bF_{\sigma}$ is a representation of the symmetric group).
Then, we define the \emph{Young symmetrizer} associated to this Young diagram on $V^{\otimes n}$ to be the linear operator
\[ \bY_{\bbf} \colonequals \sum_{\sigma \in R} \sum_{\tau \in C} \sgn(\tau) \bF_{\tau\sigma}. \]

\subsection{Irreducible representations of \texorpdfstring{$U(d)$}{U(d)}}
\label{app:irreps}

We now describe the irreps of $U(d)$, following \cite{Sengupta2012}.
This description begins with the observation that, inside $U(d)$, there is a maximal abelian subgroup of diagonal matrices
\[
    \bD_{\bm\lambda} = \begin{pmatrix}\lambda_{1} & & \\ & \ddots & \\ & & \lambda_{d}\end{pmatrix}
\]
with $\lvert\lambda_i\rvert = 1$, which is isomorphic to the $d$-dimensional torus $\mathbb{T}^d$, viewed as the $d$-fold direct product of the unit circle in $\CC$ with itself.

\subsubsection{Weights}
In this subsection we will see how the character of an irrep only depends on a vector $\bm{w} \in \mathbb{Z}^d$ which we will call the \emph{weight} of the representation. 

Consider any irrep $\pi$ of $U(d)$ on a finite dimensional vector space $V$ of dimension $m$. The linear maps $\pi(\bD_{\bm{\lambda}}) $ for $\bm{\lambda} \in \mathbb{T}^d$ commute with each other, so there is a basis of $V$ with respect to which they are diagonal,
\[
    \pi(\bD_{\bm{\lambda}}) =  \begin{pmatrix}\pi_{1}(\bD_{\bm{\lambda}}) & & \\ & \ddots & \\ & & \pi_{m}(\bD_{\bm{\lambda}})\end{pmatrix}.
\]
Each $\pi_i$ is a one-dimensional representation of $\mathbb{T}^d$, a continuous group homomorphism into $\CC$.
Such group homomorphisms are of the form
\[
    \pi_i(\bD_{\bm{\lambda}}) = \bm{\lambda}^{\bm{w}_i} = \prod_{j=1}^d \lambda_j^{w_{ij}} 
\]
for a unique $\bm{w}_i = (w_{i1}, \dots ,w_{id}) \in \mathbb{Z}^d$.
Thus, $V$ admits a decomposition into \emph{weight subspaces}, $V = V_1 \oplus \cdots \oplus V_{\ell}$, where $V_i$ is the coordinate subspace in the basis diagonalizing the $\pi(\bm\lambda)$ corresponding to the coordinates $i$ with a given weight.

We introduce the \emph{lexicographic ordering} on $\mathbb{Z}^d$, for which $\bv > \bv^{\prime}$ if the first nonzero entry in $\bv - \bv^{\prime}$ is positive.
The weight $\bw$ of $\pi$ that is largest among all the weights under this ordering is called \emph{the weight of $\pi$}.
It turns out that $\pi$ is uniquely determined by this weight, as is shown in the presentation of \cite{Sengupta2012} by verifying that the character of $\pi$ is determined by this weight.

Let us give this description of the character, since it will also be useful below.
We introduce the polynomials
\[
    a_{\bm{f}}(\bm{\lambda} ) \colonequals \sum_{\sigma \in S_d} \text{sgn}(\sigma) \prod_{i = 1}^d \lambda_{\sigma(i)}^{f_i}
    = \det \begin{pmatrix} \lambda_1^{f_1} & \dots & \lambda_d^{f_1} \\
    \vdots & & \vdots \\
    \lambda_1^{f_d} & \dots & \lambda_d^{f_d} \end{pmatrix}.
\]
We then have the following result.
\begin{Lemma}\label{lem:Ud-irrep-character}
Suppose $\pi$ is an irrep of $U(d)$ with weight $\bw = (w_1, \dots, w_d)$.
Then, the character $\chi_\pi$ of $\pi$ is the unique function on $U(d)$ that is constant on conjugacy classes and whose value on diagonal matrices is given by
\[
    \chi_\pi(\bm D_{\bm\lambda}) = \frac{a_{(w_1+d-1,\dots, w_d)}(\bm{\lambda})}{a_{(d-1,\dots,1,0)}(\bm{\lambda})}. 
\]
The right-hand side is a polynomial upon performing the polynomial division.
\end{Lemma}
\noindent
Note that the denominator is the ordinary Vandermonde determinant:
\[ a_{(d-1,\dots,1,0)}(\bm{\lambda}) = \Delta(\bm \lambda) \colonequals \det \begin{pmatrix} \lambda_1^{0} & \dots & \lambda_d^{0} \\
    \vdots & & \vdots \\
    \lambda_1^{d - 1} & \dots & \lambda_d^{d - 1} \end{pmatrix} = \prod_{1 \leq i < j \leq d}(\lambda_j - \lambda_i). \]

\subsubsection{From weights to representations}\label{sec weight}

Our next goal is to construct an irrep of $U(d)$ with a given weight $\bm{w} \in \mathbb{Z}^d_{\downarrow}$. It will be convenient to work first with a vector $\bm{f} \in \mathbb{Z}^d_{\downarrow}$ all of whose components are non-negative. We take $\bm{f}$ given by $f_j = w_j-w_d$ if this is not the case. This $\bm f$ is a partition of $n \colonequals \sum (w_j - w_d)$.

We then realize $\pi$ on a subspace $V_{\bm f} \subset (\CC^d)^{\otimes n}$, on which $\bU \in U(d)$ acts as $\bU^{\otimes n}$.
This subspace is just $V_{\bm f} = \bY_{\bf}(\CC^d)^{\otimes n}$, where $\bY_{\bm f}$ is the Young symmetrizer defined in Appendix~\ref{app:young}.
We also describe an explicit basis for this subspace below.

For $w_d < 0$, call $\pi_{\bm f}$ the representation for $\bm f$ obtained above.
We then instead consider the representation
\[
    \pi_{\bm{w}} = \pi_{\bm f}\otimes \eta^{\otimes \lvert w_d \rvert},
\]
where $\eta$ is the one-dimensional representation with $\eta(\bU) = \overline{\det(\bU)}$.
Irreducibility of $\pi_{\bm{f}}$ implies irreducibility of $\pi_{\bm{w}}$, so it suffices to check that $\pi_{\bm w}$ indeed has weight $\bm w$.

We note that $\eta$ is itself an irrep with weight $(-1, \dots, -1)$, since $\eta(\bD_{\bm\lambda}) = \lambda_1^{-1} \cdots \lambda_d^{-1}$.
Thus for $\pi_{\bm{w}}$ we have
\[
    \pi_{\bm w}(\bD_{\bm\lambda}) = \lambda_1^{f_1-w_d}\dots \lambda_d^{f_d-w_d} = \lambda_1^{w_1} \dots \lambda_d^{w_d},
\]
and thus $\pi_{\bm w}$ is an irrep with weight $\bw$.

We note in particular that, as a consequence of this discussion, the irrep with weight $\bw$ occurs as an irreducible subrepresentation of $\rho^{\otimes n} \otimes \eta^{\otimes \lvert w_d \rvert}$, where $\rho$ is the natural representation as in the main text.

\subsubsection{Explicit bases}
\label{sec:explicit-bases}

In addition to the above description, we may give quite explicit bases for $V_{\bm f}$.
Given a tableau $T$ of shape $\bm{f}$,
\begin{center}
\ytableausetup{mathmode, boxsize=3em}
\vspace{1em}
\begin{ytableau}
t_{1,1} & t_{1,2} & t_{1,3} & \none[\dots]
& \scriptstyle t_{1,f_1-1} & t_{1,f_1} \\
t_{2,1} & t_{2,2} & t_{2,3} & \none[\dots]
& t_{2,f_2} \\
\none[\vdots] & \none[\vdots]
& \none[\vdots]
\end{ytableau}
\end{center}
with $t_{i,j} \in [d]$, we can associate an element of $(\mathbb{C}^d)^{\otimes n}$ to $T$ by setting 
\begin{equation}\label{def wa}
     \be_T \colonequals \be_{t_{1,1}} \otimes \be_{t_{1,2}} \otimes   \dots \otimes \be_{t_{1,f_1}} \otimes \dots \otimes \be_{t_{k, 1}} \otimes \cdots \otimes \cdots \be_{t_{k, f_k}},
\end{equation}
taking tensor products of basis vectors across the rows of $T$.
We then have the following description.
\begin{Lemma}[Proposition 21 of \cite{sage2011explicit}] \label{lem:basis-irrep}
    The set 
    \[
        \{ \bY_{\bm f}\bm e_T : T \in \mathcal{Y}_{\bm{f}}(d) \}
    \]
    forms a basis of $V_{\bm{f}}$.  
\end{Lemma}
\noindent
Note that this result implicitly contains the standard description of $\dim(\pi_{\bm f})$ as the number of semistandard Young tableaux of shape $\bm f$ with entries in $[d]$.

A further extension of this result gives explicit bases for each individual weight subspace of $V_{\bm f}$ (note that the weight subspaces of $V_{\bw}$ are the same as those of the associated $V_{\bf}$ after shifting all weights by $-w_d$).
\begin{Lemma} \label{lem:basis-weight-subspace}
    The set 
    \[
        \{ \bY_{\bm f}\bm e_T : T \in \mathcal{Y}_{\bm{f}}^{\bm w}(d) \}
    \]
    forms a basis of the weight subspace of $V_{\bm{f}}$ having weight $\bw$. 
\end{Lemma}
\noindent
This result likewise gives a description of the dimensions of various weight subspaces as the numbers of semistandard Young tableaux with shape $\bf$ and various contents, as well as a description of the set of weights of $\pi_{\bm f}$.

\subsection{Finding a non-zero element of \texorpdfstring{$\widetilde{V}_{\bm w}$}{tildeV\_w}}
\label{app:tV-nonzero}

In the main text, we are concerned with the subspace $\widetilde{V}_{\bm w} \subseteq V_{\bm w}$ that is invariant under the action of $\TT^d \rtimes S_d < U(d)$.
As we discuss in Section~\ref{sec:subspace-Td}, the subspace invariant under $\TT^d$ is precisely the subspace associated to the weight $(0, \dots, 0)$ in the above sense.
Accordingly, Lemma~\ref{lem:basis-weight-subspace} gives that the $\bY_{\bm f} \bm e_T$ for $T$ a semistandard Young tableau of weight $\bm f$ and content $(\lvert w_d \rvert, \dots, \lvert w_d \rvert)$ (i.e., a tableau where each $i \in [d]$ occurs an equal number of times) form a basis for $\bm \Pi_{\TT^d} V_{\bm w}$.

For the computations in Section~\ref{sec:final-verification}, in the case where $\widetilde{V}_{\bm w} \neq \{\bm 0\}$, we need only find a single non-zero vector in $\widetilde{V}_{\bm w}$.
Equivalently, it suffices to find $T$ a tableau as above such that $\bm \Pi_{S_d} \bY_{\bm f} \be_T \neq \bm 0$.
Since $S_d$ is finite, the projection $\bm \Pi_{S_d}$ is given in a closed form that can be used computationally, so we simply search for such $T$ by brute force.
We remark that the above does not hold for \emph{every} $T$ of suitable shape and content, so it is indeed important to check that the appropriate projection is non-zero.

\subsection{Decomposition of tensor products of natural representations}

In the main text, we need to obtain restrictions on the irreps occurring in certain tensor product representations, which we consider here.
The \emph{natural representation} $\rho$ of $U(d)$ is the $d$-dimensional irrep with $\rho(\bU) = \bU$.
Its dual is $\rho^*(\bU) = \overline{\bU}$.
In the main text, we wish to constrain the irreps occurring in tensor products of the form
\[ \rho_{r, r^{\prime}} \colonequals \rho^{\otimes r} \otimes (\rho^*)^{\otimes r^{\prime}}, \]
for some $r, r^{\prime} \geq 0$.

More specifically, we wish to constrain the irreps occurring in \emph{all} such tensor products with $r + r^{\prime} \leq k$ for some $k$.
We recall that the reason for this is that the matrix coefficients of these tensor product representations, taken together, span $\CC_{\leq k}[\bU, \overline{\bU}]$, the space of polynomials of degree at most $k$ in the $U_{ij}$ and their conjugates.

We obtain such constraints using character theory.
The character of the tensor product representation is invariant on conjugacy classes and on diagonal matrices is given by
\[
    \tr(\rho_{r,r^{\prime}}(\bD_{\bm\lambda})) = \left(\sum_{i=1}^d \lambda_i\right)^r \left(\sum_{i=1}^d \overline{\lambda_i}\right)^{r^{\prime}}.
\]
Using Lemma~\ref{lem:Ud-irrep-character}, the multiplicity of an irrep $\pi$ with weight $\bw$ in $\rho_{r, r^{\prime}}$ is given by the inner product of characters,
\[
    m_{\bw \to r, r^{\prime}} \colonequals \int_{U(d)} \chi_\pi(\bm{U})  \overline{\chi_{\rho_{r,r'}}(\bm{U})}d\nu(\bm{U}).
\]

To compute this, we use the following general tool for integrating functions constant on conjugacy classes (this also plays an important role in the usual derivation of the character formula in Lemma~\ref{lem:Ud-irrep-character}).
\begin{Prop}[Weyl integration formula]\label{weyl int}
    Let $f$ be a continuous function on $U(d)$ that is constant on conjugacy classes. Then,
    \[
        \int_{U(d)} f(\bm{U}) d\bm{U} = \frac{1}{d!} \int_{ \mathbb{T}^d} f(\bm{\lambda})\lvert\Delta(\bm{\lambda})\rvert^2 d\bm{\lambda}.
    \]
\end{Prop}

Thus, we find
\[ m_{\bw \to r, r^{\prime}} = \frac{1}{d!} \int_{ \mathbb{T}^d} a_{(w_1+d-1, \dots, w_d)}(\bm{\lambda})\overline{\Delta(\bm\lambda)\left(\sum_{i=1}^d \lambda_i\right)^r \left(\sum_{i=1}^d \overline{\lambda_i}\right)^{r^{\prime}}} d \bm{\lambda}. \]
This can be non-zero only if some monomial is shared between the two polynomials
\[ \Delta(\bm\lambda)\left(\sum_{i=1}^d \lambda_i\right)^r \left(\sum_{i=1}^d \overline{\lambda_i}\right)^{r^{\prime}} \,\,\,\text{ and }\,\,\, a_{(w_1+d-1, \dots, w_d)}(\bm{\lambda}), \]
after canceling $\lambda_i \overline{\lambda_i} = 1$ and interpreting $\overline{\lambda_i} = \lambda_i^{-1}$ as a monomial with negative exponent.

The monomials occurring in the former have tuples of exponents of the form
\[
    \sigma(d-1, \dots, 0) + (s_1, \dots, s_d),
\]
for $\sigma \in S_d$ and the latter tuple satisfying
\begin{align}
    \sum_{i = 1}^d \max\{s_i, 0\} &\leq r, \label{eq:s-cond-1} \\
    \sum_{i = 1}^d \min\{s_i, 0\} &\geq -r^{\prime}. \label{eq:s-cond-2}
\end{align}
The monomials occurring in the latter have tuples of exponents of the form
\[
    \sigma(w_1+d-1, \dots, w_d)
\]
for $\sigma \in S_d$.
Thus, a monomial can be shared between the two only if there exist $\sigma, \tau \in S_d$ such that
\[ \bs \colonequals \sigma(w_1+d-1, \dots,w_d)-\tau(d-1, \dots, 0) \]
satisfies \eqref{eq:s-cond-1} and \eqref{eq:s-cond-2}.

The following general combinatorial result restricts when it is possible to achieve this.
\begin{Prop}
    Suppose $a_1 \geq \cdots \geq a_d$ and $b_1 \geq \cdots \geq b_d$.
    Define, for $\sigma, \tau \in S_d$,
    \begin{align*}
        F_+(\sigma, \tau) &\colonequals \sum_{i = 1}^d \max\{a_{\sigma(i)} - b_{\tau(i)}, 0\}, \\
        F_-(\sigma, \tau) &\colonequals \sum_{i = 1}^d \min\{a_{\sigma(i)} - b_{\tau(i)}, 0\}.
    \end{align*}
    Then, $F_+(\sigma, \tau)$ is minimized and $F_-(\sigma, \tau)$ is maximized when $\sigma = \tau$.
\end{Prop}

\begin{proof}
    In this proof let us write $\id$ for the identity permutation.
  Without loss of generality, we may assume $\sigma = \id$.
  Let $\tau \in S_d$ have $F_+(\id, \tau) < F_+(\id, \id)$.
  Then $\tau \neq \id$, so there exist $1 \leq i < j \leq d$ that are transposed by $\tau$, i.e., having $\tau(i) > \tau(j)$.
  Define $\tau^{\prime} \colonequals (\tau(i) \, \, \tau(j)) \circ \tau$.
  We claim that then $F_+(\id, \tau^{\prime}) \leq F_+(\id, \tau)$.
  Repeating this procedure shows eventually that $F_+(\id, \id) \leq F_+(\id, \tau)$, which is a contradiction and gives the result for $F_+$.
  For $F_-$ a symmetric argument applies.
  
  Let us write $k \colonequals \tau(j) < \tau(i) \colonequals \ell$.
  Then,
  \begin{align*}
   F_+(\id, \tau^{\prime}) - F_+(\id, \tau) 
   &= \max\{a_i - b_{k}, 0\} + \max\{a_j - b_{\ell}, 0\} \\
   &\hspace{0.8cm} - \max\{a_i - b_{\ell}, 0\} - \max\{a_j - b_{k}, 0\}.
   \end{align*}
  Enumerating all possible orderings of the four numbers $a_i \geq a_j$ and $b_k \geq b_{\ell}$ shows that this is always non-positive, completing the argument.
\end{proof}

The following results then follow immediately from applying the Proposition to our preceding computations.

\begin{Lemma}\label{weights pol}
    If an irrep of $U(d)$ with weight $\bm{w} = (w_1, \dots, w_d) \in \mathbb{Z}^d_{\downarrow}$ occurs as a subrepresentation of $\rho^{\otimes r} \otimes (\rho^*)^{\otimes r^{\prime}}$, then
    \begin{align*}
        \sum_{i = 1}^d \max\{0, w_i\} &\leq r, \\
        \sum_{i = 1}^d \min\{0, w_i\} &\geq -r^{\prime}.
    \end{align*}
\end{Lemma}

\begin{Cor}\label{weights pol degree}
    If an irrep of $U(d)$ with weight $\bm{w} = (w_1, \dots, w_d) \in \mathbb{Z}^d_{\downarrow}$ occurs as a subrepresentation of $\rho^{\otimes r} \otimes (\rho^*)^{\otimes r^{\prime}}$ for any $r, r^{\prime} \geq 0$ with $r + r^{\prime} \leq k$, then $\lvert \bw \rvert \leq k$.
\end{Cor}

\section{Details and optimizations of symbolic verification}
\label{app:computational}

We mention a few further properties that we take advantage of in the code implementing our computations to save computing time. 

Let us recall the setting of our main computation.
For $\bm w$ one of the two weights $(s, 0, 0, 0, 0, -s)$ for $s \in \{2, 3\}$, let $\bm f$ be the associated partition of $n = ds = 6s$, which in these cases will be $\bm f = (2s, s, s, s, s)$.
Following the simple brute force procedure in Appendix~\ref{app:tV-nonzero}, we find a semistandard Young tableau $T$ of shape $\bm f$ and of content $(s, s, s, s, s, s)$ (i.e., containing each index $i \in [6]$ exactly $s$ times) such that $\bv \colonequals \bm \Pi_{S_d} \bm Y_{\bm f} \bm e_T \neq \bm 0$.

Let us expand the definitions and write $\bm v$ explicitly.
We recall from Appendix~\ref{app:rep-Ud} the definitions related to $\bm Y_{\bm f}$: we have $T_0$ a standard Young tableau of the same shape as $T$, labelled by the numbers $1, 2, \dots, n$ increasing across the rows.
To $T_0$ we associate the subgroups $R, C \leq S_n$ fixing the sets of labels of each row and each column, respectively.
We then have $\bm Y_{\bm f} = \sum_{\sigma \in R} \sum_{\tau \in C} \sgn(\tau) \bm F_{\tau\sigma}$, where $\bm F_{\sigma}$ permutes the tensor axes according to $\sigma$.
Finally, the matrix $\bm \Pi_{S_d}$ averages over $\rho \in S_d$ the application of $\rho$ to \emph{every} axis of a tensor.

Let $t_1, \dots, t_n \in [d]$ be the indices written in $T$, read across the rows.
We then have
\begin{align*}
    \bm e_T &= \be_{t_1} \otimes \cdots \otimes \be_{t_n}, \\
    \bm Y_{\bm f}\bm e_T &= \sum_{\sigma \in R} \sum_{\tau \in C} \sgn(\tau) \be_{t_{( \tau \sigma)(1)}} \otimes \cdots \otimes \be_{t_{(\tau \sigma)(n)}}, \\
    \bv = \bm \Pi_{S_d}\bm Y_{\bm f}\bm e_T &= \frac{1}{d!}\sum_{\rho \in S_d} \sum_{\sigma \in R} \sum_{\tau \in C} \sgn(\tau) \be_{\rho(t_{( \tau \sigma)(1)})} \otimes \cdots \otimes \be_{\rho(t_{( \tau \sigma)(n)})}.
\end{align*}

We are interested in computing quadratic forms $\bv^{\top}\bM \bv$ for the two matrices $\bM = \pi(\bH)$ and $\bM = \bm I$.
These take the form
\begin{align*}
    \bv^{\top} \bM \bv &= \frac{1}{d!^2}\sum_{\rho, \rho^{\prime} \in S_d} \, \sum_{\sigma, \sigma^{\prime} \in R} \, \sum_{\tau, \tau^{\prime} \in C} \sgn(\tau \tau^{\prime}) \\
    &\hspace{0.1cm} (\be_{\rho(t_{( \tau \sigma)(1)})} \otimes \cdots \otimes \be_{\rho(t_{( \tau \sigma)(n)})})^{\top}\bM \be_{\rho^{\prime}(t_{( \tau^{\prime} \sigma^{\prime})(1)})} \otimes \cdots \otimes \be_{\rho^{\prime}(t_{( \tau^{\prime} \sigma^{\prime})(n)})}.
\end{align*}    

Now, defining for two tableaux $A, B$ having shape $\bm f$, entries in $[d]$, and the associated actions of $S_n$ and $S_d$, the function
\[ g(A, B) \colonequals \frac{1}{d!^2}\sum_{\rho, \rho^{\prime} \in S_d} \be_{\rho(A)}^{\top} \bM \be_{\rho^{\prime}(B)}, \]
we may rewrite the above as
\[ \bv^{\top}\bM \bv = \sum_{\sigma, \sigma^{\prime} \in R} \, \sum_{\tau, \tau^{\prime} \in C} \sgn(\tau \tau^{\prime}) g\big((\tau\sigma)(T), (\tau^{\prime}\sigma^{\prime})(T)\big). \]
When $\bM = \what{\pi}(\bH)$, then $g(A, B)$ is, up to rescaling, just the function $f(\bm i, \bm j)$ from Section~\ref{sec:final-verification}, where $\bm i$ and $\bm j$ are the entries of $A$ and $B$ read across the rows.
As with $f$, $g(A, B)$ only depends on the associated matrix $\bS^{\bm i, \bm j}$ of these vectorizations, and is unchanged by row and column permutations of this matrix.
In particular, for any $\eta \in S_n$, we have $g(A, B) = g(\eta(A), \eta(B))$.
Applying this transformation with $\eta = \tau^{\prime^{-1}}$ above, we find that we may remove one of the summations over $C$,
\[ \bv^{\top}\bM \bv = \lvert C \rvert\sum_{\sigma, \sigma^{\prime} \in R} \, \sum_{\tau  \in C} \sgn(\tau) g\big((\tau\sigma)(T), \sigma^{\prime}(T)\big). \]

For another simplification, let $R_{\mathsf{stab}} < R$ denote the subgroup of the row stabilizer that further preserves the labelling of each row of $T$ (recalling that $T$, being only semistandard, has repeated indices whereby $R_{\mathsf{stab}}$ can be non-trivial).
In the code, we abuse naming conventions slightly and refer to this as the \emph{row stabilizer of $T$}.
Let $R_1$ be a set of right coset representatives of $R_{\mathsf{stab}}$ in $R$.
The value of $g(A, \sigma(B))$ only depends on the right coset of $G$ to which $\sigma$ belongs, so we may collect these terms and find
\begin{align*}
    \bv^{\top}\bM \bv
    &= \lvert C \rvert\frac{\lvert R \rvert}{\lvert R_{\mathsf{stab}} \rvert} \sum_{\substack{\sigma \in R \\ \sigma^{\prime} \in R_1}} \, \sum_{\tau  \in C} \sgn(\tau) g\big((\tau\sigma)(T), \sigma^{\prime}(T)\big),
    \intertext{and, using $g(\tau(A), B) = g(A, \tau^{-1}(B))$, applying the same transformation to $\sigma$, and then moving $\tau$ back to the first argument of $g$ further gives}
    &= \lvert C \rvert\left(\frac{\lvert R \rvert}{\lvert R_{\mathsf{stab}} \rvert}\right)^2 \sum_{\sigma, \sigma^{\prime} \in R_1} \, \sum_{\tau  \in C} \sgn(\tau) g\big((\tau\sigma)(T), \sigma^{\prime}(T)\big).
\end{align*}

For a final simplification, we note that if $\sigma(T)$ contains a repetition of any index in any column, then the summation over $\tau$ will result in zero, since any two terms with $\tau$ differing only by a transposition of these two positions will make opposite contributions.
Thus, letting $R_2 \subseteq R_1$ denote only those elements $\sigma \in R_1$ such that $\sigma(T)$ contains no repeated elements in any column, we may again first restrict to $\sigma^{\prime} \in R_2$, and then, moving $\tau$ between the arguments of $g$ as above, also restrict to $\sigma \in R_2$.
We finally obtain:
\begin{equation*}
\bv^{\top}\bM \bv = \lvert C \rvert \left(\frac{\lvert R \rvert}{\lvert R_{\mathsf{stab}} \rvert}\right)^2 \sum_{\sigma, \sigma^{\prime} \in R_2} \, \sum_{\tau  \in C} \sgn(\tau) g\big((\tau\sigma)(T), \sigma^{\prime}(T)\big).
\end{equation*}
This reduced summation is what our code in \cite{supplementary} directly implements.
We note that the leading factor applies to both of the matrices we want to substitute for $\bM$, so it need not be included in the calculations, and likewise for the $1 / d!^2$ factor in the definition of $g$.

\end{document}